\documentclass[final,floating,letterpaper]{siamltex1213}
\usepackage{}
\usepackage{latexsym,amsmath,amssymb}
\allowdisplaybreaks[4]

\usepackage{comment}
\usepackage[compress]{cite}
\usepackage{algorithm}
\usepackage{algorithmic}

\algsetup{indent=2em}

\usepackage{mathbbold}
\usepackage{bbm}
\usepackage{amsfonts}
\usepackage[T1]{fontenc}
\usepackage{latexsym,amssymb,amsmath,amsfonts}
\usepackage{graphicx}
\usepackage{subfigure}
\usepackage{epsf}
\usepackage{epstopdf}
\usepackage{amssymb}
\usepackage{hyperref}

\usepackage{bm}

\renewcommand{\theequation}{\arabic{section}.\arabic{equation}}

\newcommand{\R}{{\mathbb R}}

\newcommand{\be}{\begin{eqnarray}}
\newcommand{\ben}{\begin{eqnarray*}}
\newcommand{\en}{\end{eqnarray}}
\newcommand{\enn}{\end{eqnarray*}}
\newcommand{\ba}{\backslash}
\newcommand{\pa}{\partial}

\newcommand{\ov}{\overline}

\newcommand{\G}{\Gamma}

\newcommand{\om}{\omega}

\newcommand{\wi}{\widetilde}

\title{Fast acoustic source imaging using multi-frequency sparse data}
\author{
Ala Alzaalig\thanks{Department of Mathematical Sciences, Michigan Technological University. Email: amalzaal@mtu.edu}
\and
Guanghui Hu\thanks{Beijing Computational Science Research Center, 100193, Beijing, China. Email: hu@csrc.ac.cn}
\and
Xiaodong Liu\thanks{Institute of Applied Mathematics, Academy of Mathematics and Systems Science, Chinese Academy of Sciences, 100190 Beijing, China.  Email: xdliu@amt.ac.cn}
\and
Jiguang Sun\thanks{Department of Mathematical Sciences, Michigan Technological University and College of Mathematical Sciences, University of Electronic Science and Technology of China. Email: jiguangs@mtu.edu}
}
\begin{document}                        
\maketitle                              

\begin{abstract}
We consider the acoustic source imaging problems using multiple frequency data. Using the data of one observation direction/point, we prove that some information (size and location) of the source support can be recovered. A non-iterative method is then proposed to image the source for the Helmholtz equation using multiple frequency far field data of one or several observation directions. The method is simple to implement and extremely fast since it only computes an indicator function on the interested domain using only matrix vector multiplications. Numerical examples are presented to validate the effectiveness of the method.
\end{abstract}



\pagestyle{myheadings}
\thispagestyle{plain}
\markboth{A. Alzaalig, G. Hu, X. Liu and J. Sun}
{Inverse source problems with broadband sparse measurements}

\section{ Introduction}

Acoustic source imaging problems have attracted attention of many researchers because of applications such as identification
of pollution source in the environment \cite{El-Badia, El-Badia2}, sound source localization \cite{Schuhmacher} and determination of the source current
distribution in the brain from boundary measurements \cite{Dassios}.

In this paper, we consider the acoustic source scattering problem.
Let $k=\om/c>0$ be the wave number of a time harmonic wave, where $\om>0$ and $c>0$ denote the frequency and sound speed, respectively.
Fixing a wave number $k_{max}>0$, we consider the wave equation with
\be\label{kassumption}
k\in(0, k_{max}).
\en
Let
\ben
D:= \bigcup_{m=1}^{M}D_{m}\subseteq  \mathbb{R}^{n}
\enn
be an ensemble of finitely many well-separated bounded domains in $ \R^{n}$, $n=2, 3,$ i.e., $\overline{D_{j}}\cap\overline{D_{l}}\ =\emptyset$ for $j\neq l$.
For any fixed $k\in(0,k_{max})$, let $F(\cdot,k)\in L^{2}(\R^n)$ represent the acoustic source with compact support $D$.
Then the  time-harmonic wave  $u\in H_{loc}^{1}( \mathbb{R}^{n})$ radiated by $F$ solves the Helmholtz equation
\be\label{Helmholtzequation}
\Delta u(x,k) +k^{2} u(x,k)= F(x,k) \quad \quad in \; \;  \mathbb{R}^{n}
\en
and satisfies the Sommerfeld radiation condition
\be\label{SRC}
\lim_{r\longrightarrow \infty} r^{\frac{n-1}{2}}\Big(\frac{\partial u}{\partial r} - iku \Big) = 0, \quad \quad r=|x|.
\en
From the Sommerfeld radiation condition \eqref{SRC}, it is well known that the scattered field $u$ has the following asymptotic behavior
\ben
u(x, k)= C(k,n)\frac{e^{ik|x|}}{|x|^{\frac{n-1}{2}}} u^{\infty}(\theta_{x}, k)+ {\cal O}(|x|^{-\frac{n+1}{2}}), \quad \theta_{x}=\frac{x}{|x|}\in S^{n-1},\;
\enn
as $r=|x|\longrightarrow \infty,$ where $C(k,n)=e^{i\pi/4}/\sqrt{8\pi k}$ \;if $n=2$ and $C(k,n)=1/4 \pi$ if $n=3.$
The complex valued function $u^\infty=u^\infty(\theta_{x}, k)$ defined on the unit sphere $S^{n-1}$
is known as the far field pattern with $\theta_{x}\in S^{n-1}$ denoting the observation direction.

The radiating solution $u$ to the scattering problem \eqref{Helmholtzequation}-\eqref{SRC} takes the form
\be\label{us}
u(x,k)=\int_{D}\Phi_{k}(x,y)F(y,k)ds(y),\quad x\in\R^{n},
\en
with
\ben\label{Phi}
\Phi_{k}(x,y):=\left\{
              \begin{array}{ll}
                \frac{i}{4}H^{(1)}_0(k|x-y|), & n=2; \\
                \frac{ik}{4\pi}h^{(1)}_0(k|x-y|)=\frac{e^{ik|x-y|}}{4\pi|x-y|}, & n=3,
              \end{array}
            \right.
\enn
being the fundamental solution of the Helmholtz equation.
Here, $H^{(1)}_0$ and $h^{(1)}_0$ are, respectively, Hankel function and spherical Hankel function of the first kind and order zero.
From asymptotic behavior of the Hankel functions, we deduce that the corresponding far field pattern has the form
\be\label{FarFieldrep}
u^{\infty}(\theta_x,k)=\int_{D}e^{-ik\theta_x\cdot\,y}F(y,k)dy,\quad \theta_x\in\,S^{n-1}.
\en

The multi-frequency ISP is to determine the source $F$ from
\begin{itemize}
  \item the scattered fields $u(x, k),\,x\in \G,\,k\in (0,k_{max})$, where $\G$ is the measurement surface that contains $D$ in its interior; or
  \item the far field patterns $u^\infty(\theta_{x}, k)$, $\theta_{x}\in S^{n-1},\,k\in (0,k_{max})$.
\end{itemize}
Most of the works in literature assume that the source $F$ is independent of the wave number $k$ and that the measurements are taken at all observation spots,
i.e., all $x\in \G$ for the scattered data and all  $\theta_{x}\in S^{n-1}$ for the far field data.
It is well known that a source with an extended support cannot be uniquely determined from measurements at a fixed frequency
\cite{DevaneyMarengoLi, DevaneySherman}.
The use of the multiple frequency data for the ISPs provides an approach to obtain a unique solution to the inverse problems \cite{BaoLinTriki, EllerValdivia}.
Actually, it can be shown that the inverse problem is uniquely solvable and is even Lipschitz stable when the highest wave number $k_{max}$ exceeds a certain real number \cite{BaoLinTriki}.
Throughout this paper, we assume that the data can be measured on a band of wave numbers, i.e., $k\in(0,k_{max})$.

In recent years, many reconstruction methods have been proposed to solve the multi-frequency ISPs.
These methods can be classified into two categories: iterative methods and non-iterative methods.
While very successful in many cases, iterative methods are usually computationally expensive since they require the solution of a direct problem in every step \cite{BaoLuRundellXu, ZhangSun2015JCP}.
In contrast, the second group of reconstruction methods, i.e., non-iterative methods, avoids this problem, e.g., \cite{Potthast2010IP, Griesmaier2011IP,Sun2012IP, LiuSun2014IP,WangEtal2017IP, Griesmaier}.

In this work, we consider the case of a more general source $F$, which may depend on the wave number.
In particular, we are interested in broadband sparse measurements.
For the case of near field measurements, we assume that the scattered field can only be measured at finitely many points,
\ben
\{x_1, x_2, \cdots, x_M\}=:\G_M\subset \G.
\enn
Accordingly, we obtain the following broadband sparse near field measurements
\be\label{eqivalentdatanear}
\{u(x_{m}, k)\,|\, x_m\in\G_M,\,\, k\in (0,k_{max})\}.
\en
For the case of far field measurements, we suppose that the far field data can only be measured in finitely many
observation directions,
\ben
\{\theta_1, \theta_2, \cdots, \theta_M\}=:\Theta_{M}\subset S^{n-1}.
\enn
Consequently, we obtain the following broadband sparse far field measurements
\be\label{eqivalentdatafar}
\{u^{\infty}(\theta_{m}, k)\,|\, \theta_m\in\,\Theta_{M},\,\, k\in (0,k_{max})\}.
\en

The inverse problem that we consider here is to deduce information on the support for the source $F$ from these data in \eqref{eqivalentdatanear} or \eqref{eqivalentdatafar}.
It is impossible to uniquely determine the source, see e.g., the counter-example \eqref{f1}-\eqref{f2} given in Section 2.2.
Nevertheless, we will show that partial information of the source, e.g., the location and the support, can be
approximately reconstructed from these data. Actually, based on the broadband sparse far field measurements, Sylvester and Kelly \cite{Sylvester} produced a convex polygon with
normals in the $\theta_m$ directions containing the source.
Such a reconstruction is done for those sources that are independent of the wave number.
The first contribution of this paper is to analyze the inverse problem with more general source that may depend on the wave number using both the near field and far field measurements.

The second contribution of this paper is to propose some novel non-iterative sampling methods for source support reconstruction with broadband sparse far field data.
These methods are very fast and simple to implement because the indicator functions behind them are based on the inner product only.
If the far field data can only be obtained in a single observation direction, we introduce an indicator to reconstruct a strip containing the support of
the source with the observation direction as its normal direction.
Thus, two linearly independent observation directions give a rough reconstruction for the support of the source.

The rest of this paper is arranged as follows. Some uniqueness results will be established in Section 2 by using multi-frequency data with a single measurement point.
Section 3 is devoted to a novel indicator function using complete far field data. We will show that the indicator decays as Bessel functions when sampling points are away from the source support.
In Section 4, we consider the cases when the far field data can be measured for finitely many observation directions.
In particular, the behavior of the indicator for a single observation direction will be established.
Section 5 contains several numerical examples in two dimensions to verify the effectiveness and efficiency of our method.

\section{Uniqueness results from multi-frequency data with a single measurement point}
This section is devoted to the study of what information can be obtained using multi-frequency data with a single measurement point.
Specifically, the measurements are either the broadband near field data $u(x,k)$, $k\in (0, k_{max})$ for a fixed measurement point $x \in\G_M$, or the broadband far field data
$u^{\infty}(\theta, k),$ $k\in (0, k_{max})$ with a single observation direction $\theta\in\,\Theta_{M}$.
For a bounded domain $D$,
the $x$-annular hull for a single measurement point $x \in \G_M$ is defined by
\ben
A_{D}(x):=  \{ y\in \mathbb \R^{n}\; | \; \inf_{z\in D}|x-z| \leq |x-y|\leq \sup_{z\in D}|x-z|\},
\enn
which is the smallest annular (difference of two concentric discs/balls) with center at the measurement point $x \in \G_M$,
while the $\theta$-strip hull of $D$ for a single observation direction $\theta \in \,\Theta_{M}$ is defined by
\ben
S_{D}(\theta):=  \{ y\in \mathbb \R^{n}\; | \; \inf_{z\in D}z\cdot \theta \leq y\cdot \theta \leq \sup_{z\in D}z\cdot \theta\},
\enn
which is the smallest strip (region between two parallel hyper-planes) with normals in the directions $\pm \theta$ that contains $\overline{D}$.
See Figure \ref{SDAD} for these two hulls in two dimensions.

\begin{figure}
\begin{center}
\resizebox{0.45\textwidth}{!}{\includegraphics{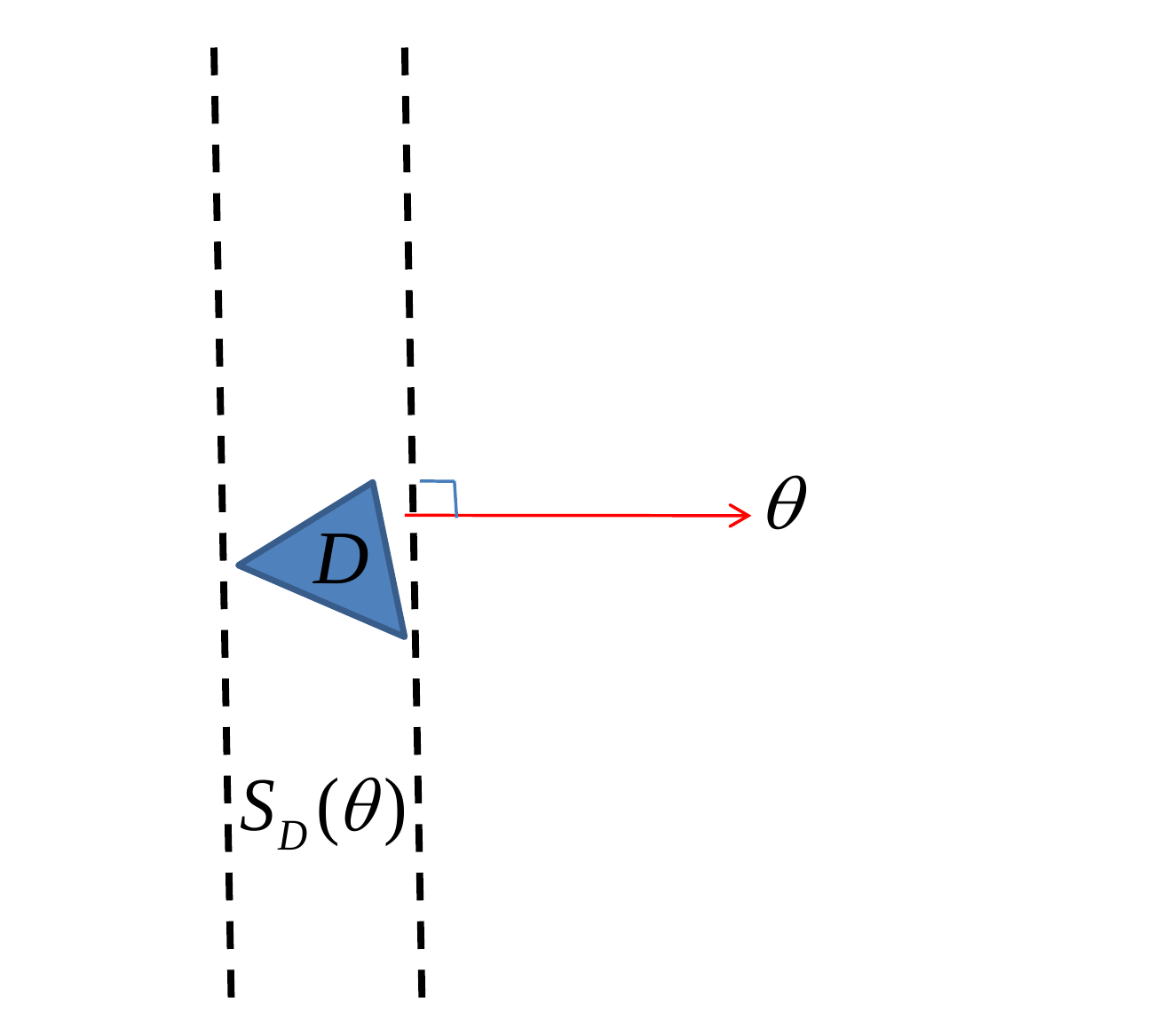}}
\resizebox{0.45\textwidth}{!}{\includegraphics{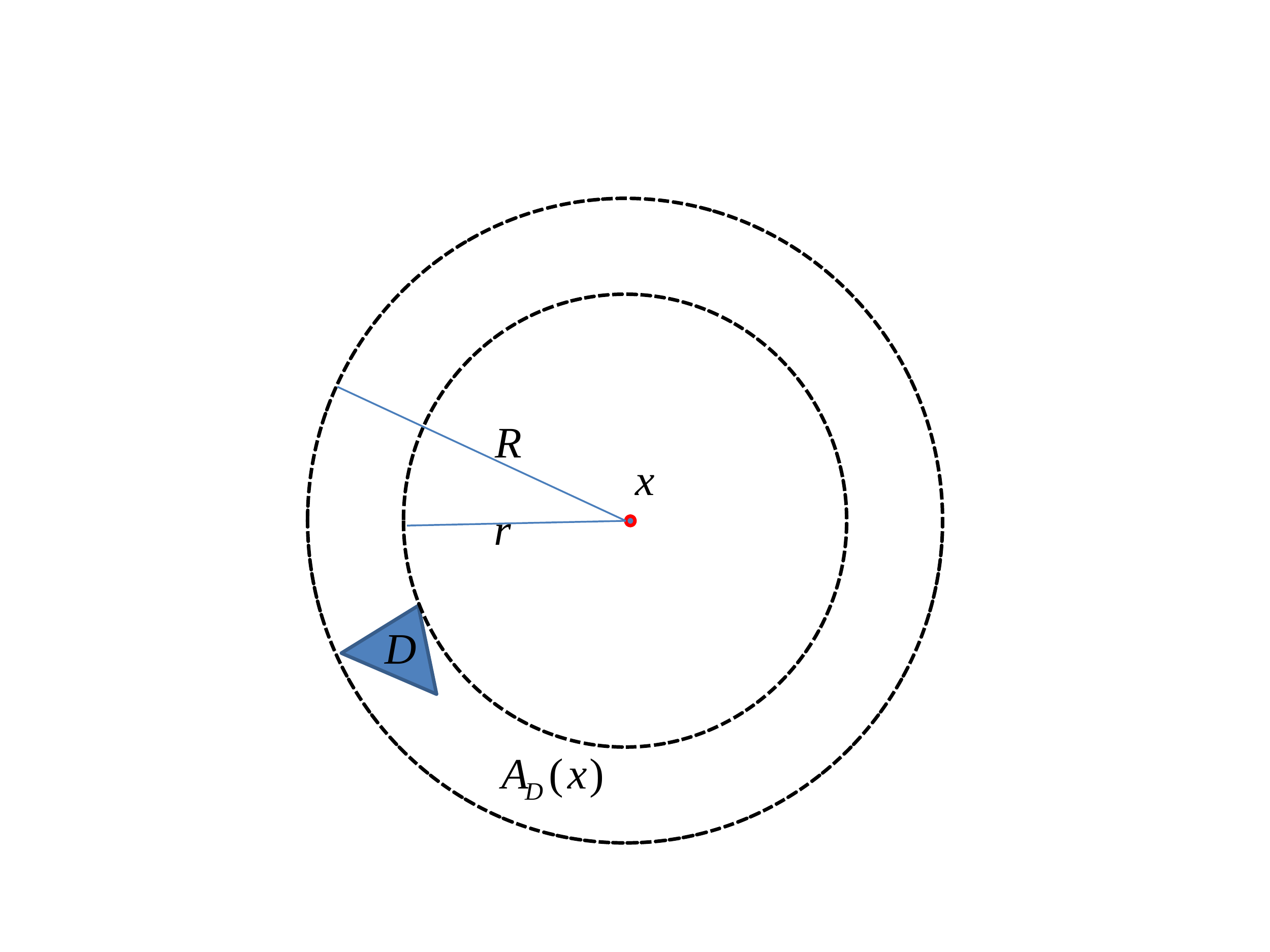}}
\end{center}
\caption{One object in $\R^2$. Left: the smallest strip $S_D(\theta)$ with normal in the observation direction $\theta$;
                    Right: the smallest annular $A_D(x)$ with center at the measurement point $x$.}
\label{SDAD}
\end{figure}

We assume that the source term $F(x,k)$ takes one of the following forms
\begin{itemize}
  \item $F(x,k)=f(x)g(k)$ with a function $g\in C(0,k_{max})$;
  \item $F(x,k)=f(x)\frac{e^{ik|x-z|}}{4\pi|x-z|}$ with a point $z\in \R^3\ba\ov{D}$;
  \item $F(x,k)=f(x)e^{ikx\cdot d}$ with a direction $d\in S^{n-1}$.
\end{itemize}

For the case of near field measurements, we can only prove the uniqueness result in $\R^3$. Our method seems to fail for the two dimensional case since the proof relies on the
Fourier transform, in which a representation in the form of exponential function is needed.
We remark that the last two cases also arise in the Born approximation for scattering from an inhomogeneous medium. Indeed, the mathematical model
for scattering by an inhomogeneous medium is
\ben
\Delta u^t+k^2qu^t=0.
\enn
The total field $u^t$ is in the form $u^t=u^i+u^s$,  where $u^i$, the incident wave, is a solution to the homogeneous Helmholtz equation
and $u^s$, the radiating scattered wave, is the unique outgoing solution to
\be\label{usmedium}
\Delta u^s+k^2u^s=k^2(1-q)(u^i+u^s).
\en
The incident wave $u^i$ of particular interest is either a plane wave $e^{ikx\cdot\theta}$ with $\theta\in S^{n-1}$ being the direction of propagation, or a point
source $\frac{e^{ik|x-z|}}{4\pi|x-z|}, x\neq z$.
The Born approximation is obtained when $u^s$ is small enough to be ignored on the right hand side of \eqref{usmedium}. Thus, the Born approximation $u_B$ to $u^s$ satisfies
\be\label{usborn}
\Delta u_B+k^2u_B=k^2(1-q)u^i.
\en
Dividing \eqref{usborn} by $k^2$, the right hand side of the equation \eqref{usborn} reduces to $F(x,k)$ considered
in the last two cases by taking different incident waves.

\subsection{$F(x,k)=f(x)g(k)$ with a given function $g\in C(0,k_{max})$}\quad\;\\

First, consider the case when the source $F(x,k)$ is a product of a spatial function $f$ and a frequency function $g$.
Here, $g\in C(0,k_{max})$ is a given nontrivial function of $k$. Thus, there exists an interval $I:=(a,b)\subset(0,k_{max})$
such that
\be\label{gassumption}
g(k)\neq 0, \quad k\in I\subset(0,k_{max}).
\en

For the near field case in $\R^3$, we assume the scattered field $u(x,k)$ is measured for all $k\in I$, but only at a fixed point $x\in\R^3\ba\ov{D}$.
Let $S_r(x)$ be the sphere with radius $r$ centered at $x$.
Recalling the solution representation \eqref{us}, we obtain that
\be\label{unifgN}
v(x,k):=\frac{u(x,k)}{g(k)}
=\int_{\R^3}f(y)\frac{e^{ik|x-y|}}{4\pi|x-y|}dy
=\int_{-\infty}^{\infty}\wi{f}(r)e^{ikr}dr,\quad k\in I,
\en
where
\be
\wi{f}(r):=\left\{
             \begin{array}{ll}
               \frac{1}{4\pi r}\int_{S_r(x)}f(y)ds(y), & \hbox{$r>0$;} \\
               0, & \hbox{$r\leq0$.}
             \end{array}
           \right.
\en

The following theorem shows that multi-frequency data can be used to determine an annular containing
the support of the target.

\begin{theorem}\label{theorem1}
Assume that $F(x,k)=f(x)g(k)$, where $g$ is a given continuous function satisfying \eqref{gassumption}, and that the set
\be\label{set1}
\{r\in (0,\infty)|\, S_r(x)\subset A_D(x), \wi{f}(r)=0\}
\en
has Lebesgue measure zero.
Then, the annular $A_D(x)$ can be uniquely determined by the scattered field $u(x,k)$ for all $k\in I$, but a fixed point $x\in\R^3\ba\ov{D}$.
\end{theorem}
\begin{proof}
We firstly obtain from \eqref{unifgN} that $v(x, k)$ is an analytic function on the wave number $k$. Thus we have the data $v(x,k)$ for all $k\in \R$ by analyticity.
It is evident from \eqref{unifgN} that the data $v(x,k)$ is just the inverse Fourier transform of $\wi{f}$. Using Fourier transform, we deduce that $\wi{f}$ is uniquely determined by
the measurements $u(x,k)$ for all $k\in I$ at a fixed point $x\in\R^3\ba\ov{D}$.
Since the set in \eqref{set1} has Legesgue measure zero, we have
\ben
A_D(x)=\ov{\bigcup_{r\in (0,\infty)} \{S_r(x)| \wi{f}(r)\neq 0\}},
\enn
which implies the annular $A_D(x)$ is uniquely determined by $\wi{f}$, and also by $u(x,k)$ for all $k\in I$, but a fixed point $x\in\R^3\ba\ov{D}$.
The proof is complete.
\end{proof}

Now we turn to the far field measurements $u^{\infty}(\theta, k)$ for all $k\in (0,\infty)$, but a fixed observation direction $\theta\in S^{n-1}$.
Let
\[
\Pi_{\tau}:=\{y\in \R^n | y\cdot\theta+\tau=0\}
\]
be a hyperplane with normal $\theta$.
From the representation \eqref{FarFieldrep} for the far field patterns, by noting again the fact that $g(k)\neq 0$ for $k\in I$, we have
\be
v^{\infty}(\theta,k):=\frac{u^{\infty}(\theta,k)}{g(k)}=\int_{\R^n}e^{-ik\theta\cdot\,y}f(y)dy=\int_{\R}\hat{f}(\tau)e^{ik\tau}d\tau,\, \theta\in\,S^{n-1}, \quad
\en
where
\be
\hat{f}(\tau):=\int_{\Pi_\tau}f(y)ds(y).
\en
The analogous result of Theorem \ref{theorem1} is formulated in the following theorem for far field measurements.
\begin{theorem}\label{theorem2}
Assume $F(x,k)=f(x)g(k)$ with a given continuous function $g$ satisfying \eqref{gassumption}. If the set
\be\label{set2}
\{\tau\in\R|\, \Pi_\tau\subset S_D(\theta), \hat{f}(\tau)=0\}
\en
has Lebesgue measure zero,
then the strip $S_D(\theta)$ can be uniquely determined by the far field measurements $u^{\infty}(\theta,k)$ for all $k\in I$, but a fixed observation direction $\theta\in S^{n-1}$.
\end{theorem}

The proof of Theorem \ref{theorem2} is quite similar to that of Theorem \ref{theorem1}. However, the uniqueness result using far field measurements holds both in two and three
dimensions.

Finally, we give some remarks on the assumptions on the sets in \eqref{set1} and \eqref{set2}.

\begin{itemize}
  \item The sets in \eqref{set1} and \eqref{set2} have Lebesgue measure zero if the real part of a complex multiple of the spatial function $f$
 is bounded away from zero on their support, i.e., we assume that $f\in L^{\infty}(D)$ is such that there exist $\alpha \in \R$ and $c_{0} >0$ such that
\be\label{fassumption}
\Re(e^{i\alpha}f(x))\geq c_{0},\quad a.e.\,\, x\in D.
\en
  \item Theorems \ref{theorem1} and \ref{theorem2} are not true in general if the sets in \eqref{set1} and \eqref{set2}, respectively, have positive Lebesgue measure.
  For example, for $x=(x_1,x_2)\in\R^2$,
  we consider
  \be\label{f1}
  f_1(x)=\left\{
         \begin{array}{ll}
           1, & \hbox{$x_1\in (-1,1), x_2\in [1,2)$;} \\
           x_1, & \hbox{$x_1\in (-1,1), x_2\in (-1,1)$;} \\
           1, & \hbox{$x_1\in (-1,1), x_2\in (-2,-1]$;} \\
           0, & \hbox{otherwise.}
         \end{array}
       \right.
  \en
  and
  \be\label{f2}
  f_2(x)=\left\{
         \begin{array}{ll}
           1, & \hbox{$x_1\in (-1,1), x_2\in [1,2)$;} \\
           1, & \hbox{$x_1\in (-1,1), x_2\in (-2,-1]$;} \\
           0, & \hbox{otherwise.}
         \end{array}
       \right.
  \en
  Then $f_1\in L^2(\R^2)$ with compact support in $D_1=[-1,1]\times[-2,2]$, and $f_2\in L^2(\R^2)$ with compact support in $D_2:=D_2^{(1)}\cup D_2^{(2)}$, where
  $D_2^{(1)}=[-1,1]\times[-2,-1]$ and $D_2^{(2)}=[-1,1]\times[1,2]$. 
  A straightforward calculations shows that the corresponding far field patterns corresponding
  to these two different sources coincide for all wave numbers at the fixed observation direction $\theta=(0,1)$. We also refer to Figure \ref{LargeOneF1} for the corresponding
  numerical result.
\end{itemize}

\subsection{$F(x,k)=f(x)\frac{e^{ik|x-z|}}{4\pi|x-z|}$ with a given point $z\in \R^3\ba\ov{D}$}\quad\;\\

We consider the second case, i.e., the source $F$ takes the form
\ben
F(x,k)=f(x)\frac{e^{ik|x-z|}}{4\pi|x-z|}, \quad x\in\R^3\ba\{z\}, z\in \R^3\ba\ov{D}.
\enn
Let $S^{\prime}_r(x,z):=\{y\in \R^3 | |x-y|+|y-z|= r\}$ be an ellipsoid and $A^{\prime}_D(x,z)$ be the smallest annular-like domain, i.e., the difference of two ellipsoids,
such that $D$ is contained in its interior.
Recalling the solution representation \eqref{us}, we have that
\be
u(x,k)
=\int_{\R^3}f(y)\frac{e^{ik|x-y|}}{4\pi|x-y|}dy
=\int_{\R}\wi{f}^{\prime}(r)e^{ikr}dr,\quad k\in I,
\en
where
\be
\wi{f}^{\prime}(r):=\left\{
             \begin{array}{ll}
               \int_{S^{\prime}_r(x,z)}\frac{f(y)}{4\pi |x-y||z-y|}ds(y), & \hbox{$r>|x-z|$;} \\
               0, & \hbox{$r\leq |x-z|$.}
             \end{array}
           \right.
\en
Similar to the case considered in Theorem \ref{theorem1}, we have the following uniqueness result.
\begin{theorem}\label{theorem3}
Assume $F(x,k)=f(x)\frac{e^{ik|x-z|}}{4\pi|x-z|}$ with a given point $z\in \R^3\ba\ov{D}$. If the set
\be\label{set3}
\{r\in (0,\infty)|\, S^{\prime}_r(x,z)\subset A^{\prime}_D(x,z), \wi{f}^{\prime}(r)=0\}
\en
has Lebesgue measure zero.
Then the annular-like domain $A^{\prime}_D(x,z)$ can be uniquely determined by the scattered field $u(x,k)$ for all $k\in (0,k_{max})$, but a fixed point $x\in\R^3\ba\ov{D}, \,x\neq z$.
\end{theorem}

\subsection{$F(x,k)=f(x)e^{ikx\cdot d}$ with a given direction $d\in S^{n-1}$}\quad\;\\

Considering now the case when $F(x,k)=f(x)e^{ikx\cdot d}$ with a given direction $d\in S^{n-1}$.
Recalling again the representation \eqref{FarFieldrep} of the far field pattern, we have, for $\theta\neq d$,
\be
u^{\infty}(\theta,k)=\int_{\R^n}e^{-ik(\theta-d)\cdot\,y}f(y)dy,\quad d, \theta\in\,S^{n-1},\, k\in (0,k_{max}).
\en
Define $\theta_d:=(\theta-d)/|\theta-d|\in S^{n-1}$. The inverse problem is equivalent to seek the source $F(x,k)=f(x)$ from the far field measurements
$u^{\infty}(\theta_d,k)$ for all $k\in (0, k_{max}|\theta-d|)$. From Theorem \ref{theorem2}, we have the following uniqueness result.
\begin{theorem}\label{theorem4}
Assume $F(x,k)=f(x)e^{ikx\cdot d}$ with a given direction $d\in S^{n-1}$. If the set
\be\label{set4}
\{\tau\in\R|\, \Pi_\tau\subset S_D(\theta-d), \hat{f}(\tau)=0\}
\en
has Lebesgue measure zero.
Then the strip $S_D(\theta-d)$ can be uniquely determined by the far field measurements $u^{\infty}(\theta,k)$ for all $k\in (0,k_{max})$, but a fixed observation direction $\theta\in S^{n-1}$.
\end{theorem}

\section{A novel indicator function with complete far field data}

We begin with the inverse source problem using far field patterns $u^{\infty}(\theta_{x}, k)$ for all observation directions $\theta_x\in S^{n-1}$ and all wave numbers
$k\in (0,k_{max})$.
To reconstruct the support of a general source $F(x,k)$, we introduce the following indicator
\be\label{indicatorfull}
I(z)= \Bigg|\int_{0}^{k_{max}} \int_{S^{n-1}} u^{\infty}(\theta_x,k)\,e^{ik\theta_x\cdot\, z} ds(\theta_x)\; dk\Bigg|, \quad z\in  \mathbb{R}^{n}.
\en
Note that no a priori information on the source is required and the indicator is simple to implement since only integral evaluations are needed.
Furthermore, we will show later that the indicator continuously depends on the noise in the data.

Inserting the representation \eqref{FarFieldrep} into the indicator \eqref{indicatorfull}, changing the order of the integration,
and using the Funk-Hecke formula \cite{LiuIP17},
we deduce that
\be
I(z)
&=&\Bigg|\int_{0}^{k_{max}} \int_{S^{n-1}} u^{\infty}(\theta_x,k)\,e^{ik\theta_x\cdot\, z} ds(\theta_x)\; dk\Bigg|\cr
&=&\Bigg|\int_{0}^{k_{max}} \int_{S^{n-1}} \int_{D}e^{-ik\theta_x\cdot\,y}F(y,k)dy\,e^{ik\theta_x\cdot\, z} ds(\theta_x)\; dk\Bigg|\cr
&=&\Bigg|\int_{0}^{k_{max}} \int_{D} \int_{S^{n-1}}e^{-ik\theta_x\cdot\,y}\,e^{ik\theta_x\cdot\, z} ds(\theta_x)\; F(y,k)dy\; dk\Bigg|\cr
&=&\Bigg|\int_{0}^{k_{max}} \mu \int_{D}g(k|y-z|)F(y,k)dy\; dk\Bigg|,
\en
where
\be\label{mu}
\qquad \mu=\left\{ \begin{array}{rcl}
 2\pi, & n=2 ;\\
 4\pi, & n=3,
\end{array}\right.
  \quad\mbox{and}\quad
  g(t)=\left\{ \begin{array}{rcl}
 J_{0}(t), & n=2 ;\\
 j_{0}(t), & n=3.
\end{array}\right.
\en
Here, $J_{0}$ and $ j_{0}$  are the Bessel function and spherical Bessel function of order zero, respectively.
This implies that the indicator function $I(z)$ is a superposition of the Bessel functions in {\rm 2D} and spherical Bessel functions in {\rm 3D}, although
we do not know the coefficients which depend on the source $F$ in $D$.
We have the following asymptotic formulas for the Bessel and spherical Bessel functions
\ben
J_{0}(t)&=&\frac{\sin t +\cos t}{\sqrt{\pi t}} \left\{  1+O\Big(\frac{1}{t}\Big)\right\}, \, \, as\, \, t  \rightarrow  \infty, \\
j_{0}(t)&=&\frac{\sin t}{t} \left\{  1+O\Big(\frac{1}{t}\Big)\right\}, \, \, as\, \, t  \rightarrow  \infty.
\enn
See Figure \ref{falphadecay} for the behavior of these two functions.
\begin{figure}
\begin{center}
\resizebox{0.45\textwidth}{!}{\includegraphics{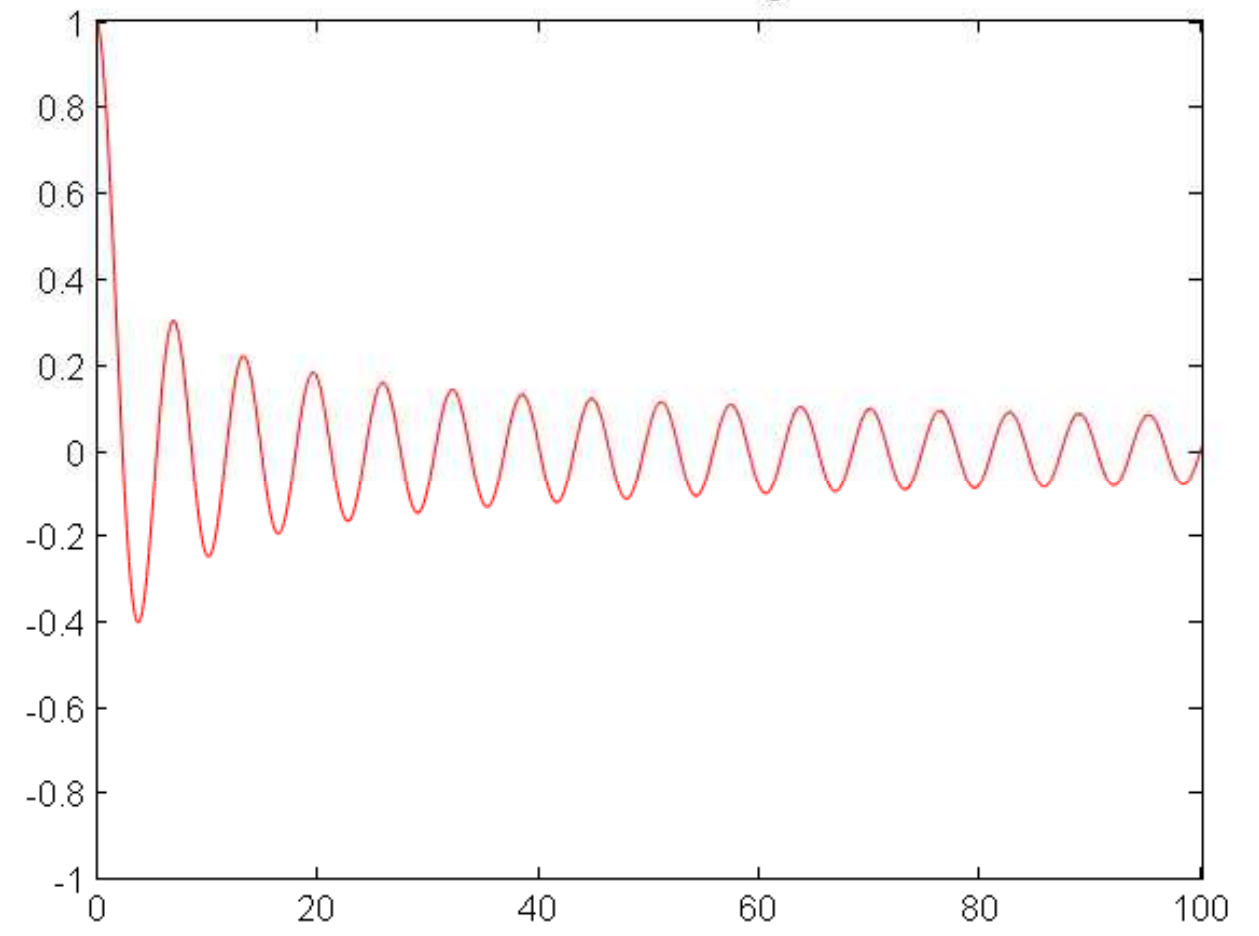}}
\resizebox{0.45\textwidth}{!}{\includegraphics{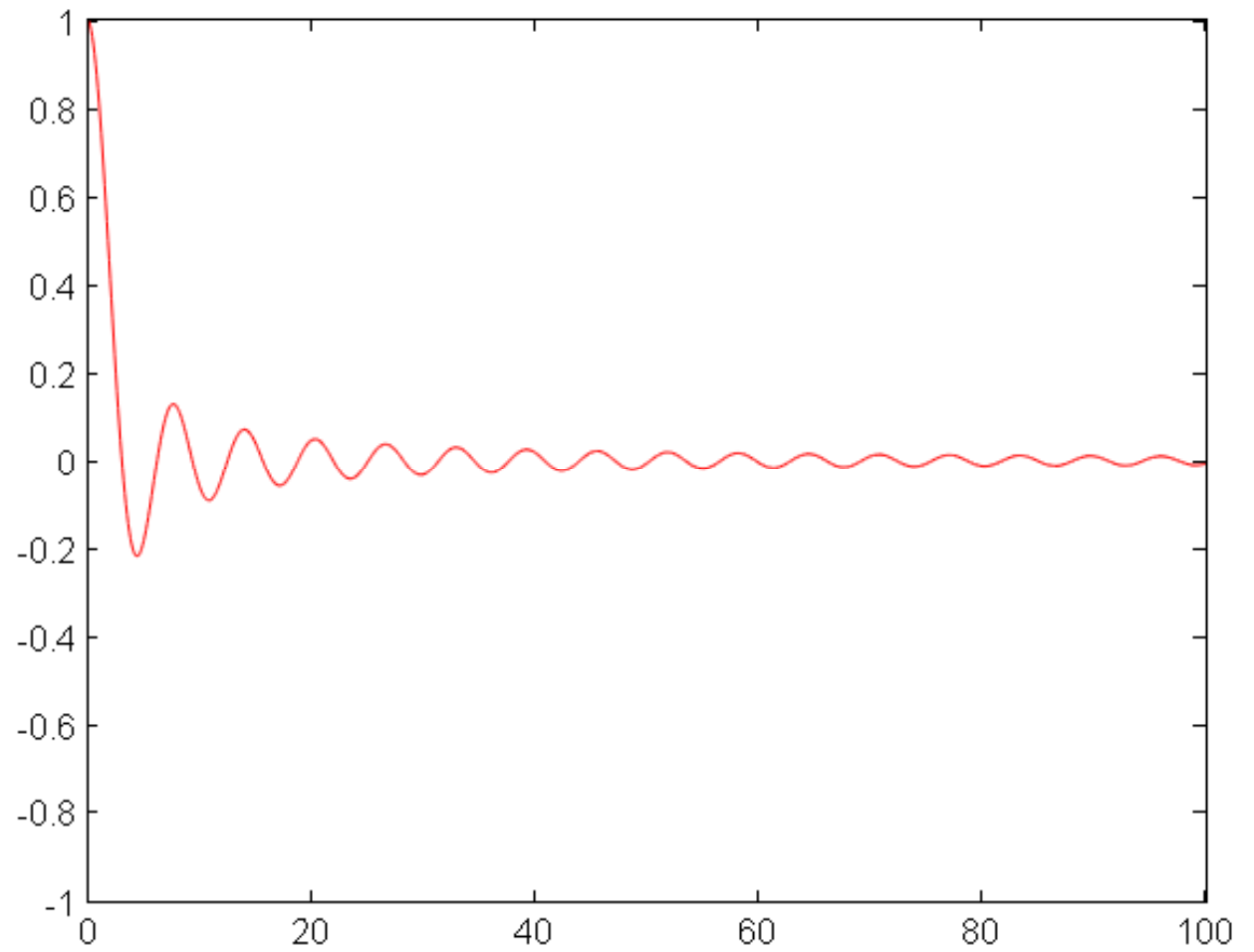}}
\end{center}
\caption{Decay behavior of the Bessel function $J_{0}$ (Left) and the spherical Bessel function $j_{0}$ (Right) .}
\label{falphadecay}
\end{figure}
{\em\bf Thus, we expect that the indicator function $I_{z}$ decays as the Bessel functions when the sampling point $z$ moves away from the support $D$.}

We end this section by a stability statement, which states that the indicator depends on the noise continuously.

\begin{theorem}{\rm (Stability statement).}
Let $u^{\infty}_{\delta}$ be the measured far field pattern with noise.  Then
\be
 |I(z)-I_{\delta}(z)|\leq  \;\mu \int_{0}^{k_{max}}\Big\| u^{\infty}(\cdot, k)-u^{\infty}_{\delta}(\cdot, k)\Big\|_{L^2(S^{n-1})}  dk,
\en
where $I_{\delta}(z)$ is the indicator functional with $u^{\infty}$ replaced by $u_{\delta}^{\infty}$, $\mu$ is a constant given by \eqref{mu}.
\end{theorem}
\begin{proof}
\ben
&&I(z)-I_{\delta}(z)\cr
&:=&\Big|\int_{0}^{k_{max}}\int_{S^{n-1}}u^\infty(\theta_x,k)e^{ik\theta_x\cdot z}ds(\theta_x)dk\Big|\cr
&& \qquad \qquad -\Big|\int_{0}^{k_{max}}\int_{S^{n-1}}u_{\delta}^\infty(\theta_x,k)e^{ik\theta_x\cdot z}ds(\theta_x)dk\Big|\cr\vspace{2cm}
&\leq&\int_{0}^{k_{max}}\int_{S^{n-1}}\Big|[u^\infty(\theta_x,k)-u_{\delta}^\infty(\theta_x,k)]e^{ik\theta_x\cdot z}\Big|ds(\theta_x)dk\cr
&\leq&\mu\int_{0}^{k_{max}}\Big\|u^\infty(\cdot,k)-u_{\delta}^\infty(\cdot,k)\Big\|_{L^{2}(S^{n-1})}dk,
\enn
where we have used the triangle inequality and the Cauchy-Schwarz inequality.
\end{proof}

\section{Multi-frequency ISPs with sparse far field data}
This section is devoted to a reconstruction method that efficiently utilizes multi-frequency information to reduce the number of sensors
and obtain a useful reconstruction of the support of the source.

For a single observation direction $\theta_m\in\Theta_M$, we introduce the following indicator
\be\label{indicator1}
I^{\theta_m}(z)= \Big|\int_{0}^{k_{max}}u^{\infty}(\theta_m, k)e^{ik\theta_m\cdot\,z}dk\Big|.
\en
While for all the available observation directions in $\Theta_M$, we introduce the indicator
\be\label{indicatorM}
I^{\Theta_M}(z):=\sum_{\theta_m\in\Theta_M}I^{\theta_m}(z).
\en

We begin with the behavior of the indicator $I^{\theta_m}$, which uses the multi-frequency far field patterns $u^{\infty}(\theta_m,k)$
with a single observation direction $\theta_m\in S^{n-1}$.
Let $\theta_m^{\perp}$ be a direction perpendicular to $\theta_m$ and we have
\be\label{I1behavior1}
I^{\theta_m}(z+\alpha\theta_m^{\perp})=I^{\theta_m}(z), \quad \forall z\in\R^n, \,\alpha\geq0,
\en
since $\theta_m^{\perp}\cdot \theta_m=0$.
This further implies that the indicator $I^{\theta_m}$ has the same value for sampling points in the hyperplane with normal direction $\theta_m$.

In the following, we assume that the source $F$ depends smoothly on the wave number $k$, say $C^1$.
Recall from \eqref{FarFieldrep} that the far field pattern has the following representation
\ben
\qquad u^{\infty}(\theta_m,k)=\int_{D}e^{-ik\theta_m\cdot\,y}f(y,k)dy,\quad\theta_m\in\Theta_M, \,k\in (0,k_{max}).
\enn
Inserting it into the indicator $I^{\theta_m}$ defined in \eqref{indicator1}, changing the order of integration, and integrating by parts, we have
\be\label{I1behavior2}
I^{\theta_m}(z)
&=&\left|\int_{D}\int_{0}^{k_{max}}e^{ik\theta_m\cdot\,(z-y)} F(y,k)dk dy \right|\cr
&=&\left|\int_{D}\frac{\mathcal {F}_z(y)}{\theta_m\cdot\,(z-y)}dy\right|,
\en
where the numerator $\mathcal {F}_z(y)\in L^{\infty}(D)$ is given by
\ben
\mathcal {F}_z(y):=e^{ik_{max}\theta_m\cdot\,(z-y)}F(y,k_{max})-F(y,0)-\int_{0}^{k_{max}}e^{ik\theta_m\cdot\,(z-y)}\frac{\pa F}{\pa k}(y,k)dk, \, y\in D.
\enn
It is clear that the indicator $I^{\theta_m}$ is a superposition of functions that decay as $1/|\theta_m\cdot\,(z-y)|$ as the sampling point $z$ goes away
from the strip $S_{D}(\theta_m)$.

In conclusion, the indicator $I^{\theta_m}$ has the same value for sampling points in the hyperplane with normal direction $\theta_m$,
and decays as the sampling point $z$ moves away from the strip $S_{D}(\theta_m)$. Thus, we expect that the value of the indicator $I^{\theta_m}$ is large inside the strip $S_{D}(\theta_m)$
and decays for sampling points moving away along the directions $\pm\theta_m$. Therefore, a strip $S_{D}(\theta_m)$ can be reconstructed using $I^{\theta_m}$.
A natural idea is to use the indicator $I^{\Theta_M}$ given in \eqref{indicatorM} to construct the so called $\Theta_M$-convex hull of $D$ \cite{Sylvester}
\ben
S_D(\Theta_M):=  \bigcap_{\theta_m\in \Theta_M}S_{D}(\theta_m).
\enn

\section{Numerical examples}
Now we present some examples to demonstrate the indicators proposed in the previous section
 in two dimensions using multiple frequency far field data. The synthetic data of the forward problems are computed using the equation \eqref{FarFieldrep}.
 Let $D$ be the compact support of $F$. We generate a triangular mesh ${\mathcal T}$ for $D$ with mesh size $h \approx 0.01$. For the direction
 \ben
 \theta:=(\cos\varphi,\sin\varphi)
 \enn
 and fixed $k$,
 the far-field pattern is approximated by
 \begin{equation}\label{uinftyh}
 u^{\infty}(\theta; k) \approx \sum_{T \in {\mathcal T}} e^{-ik \theta \cdot {y_T}} f(y_T) |T|,
 \end{equation}
 where $T \in {\mathcal T}$ is a triangle, $y_T$ is the center of $T$, and $|T|$ denotes the area of $T$.

For all examples, for $\theta_m \in \Theta_M$, we assume to have multiple frequency far field data
\[
u^{\infty}(\theta_m; k_{j}), \quad j=1, \, . \, . \, .\, N,
\]
where $N=20$, $k_{min}=0.5$, $k_{max}=20$ such that
\[
k_{j}:=(j-0.5)\Delta k,  \quad \Delta k:=\frac{k_{max}}{N}.
\]

Assume that the sampling domain is $\mathbb{S}:=[-4, 4] \times [-4, 4]$. The interval $[-4,4]$ is uniformly divided into $80$ intervals and we end up with $81 \times 81$ sampling points
uniformly distributed in $\mathbb{S}$. We denote by $Z$ the set of all sampling points.

\subsection{One observation direction}
We first consider the case of one observation direction $\theta\in\Theta_M$. We normalize the indicator function, i.e.,
the plot is for $I^{\theta}/M(I^\theta)$ where $M(I^\theta)$ is the largest element of $I^{\theta}(z), z \in Z$. Let $F=5$ and the support of $F$ is a rectangle given by $(1,2) \times (1, 1.6)$.
 In Fig.~\ref{OneOneF1}, we plot the indicators for four different observation angles $\varphi=-\pi/4, 0, \pi/8$ and $\pi/2$.
 The picture clearly shows that the source lies in a strip, which is perpendicular to the observation direction.
\begin{figure}
\begin{center}
\begin{tabular}{cc}
\resizebox{0.5\textwidth}{!}{\includegraphics{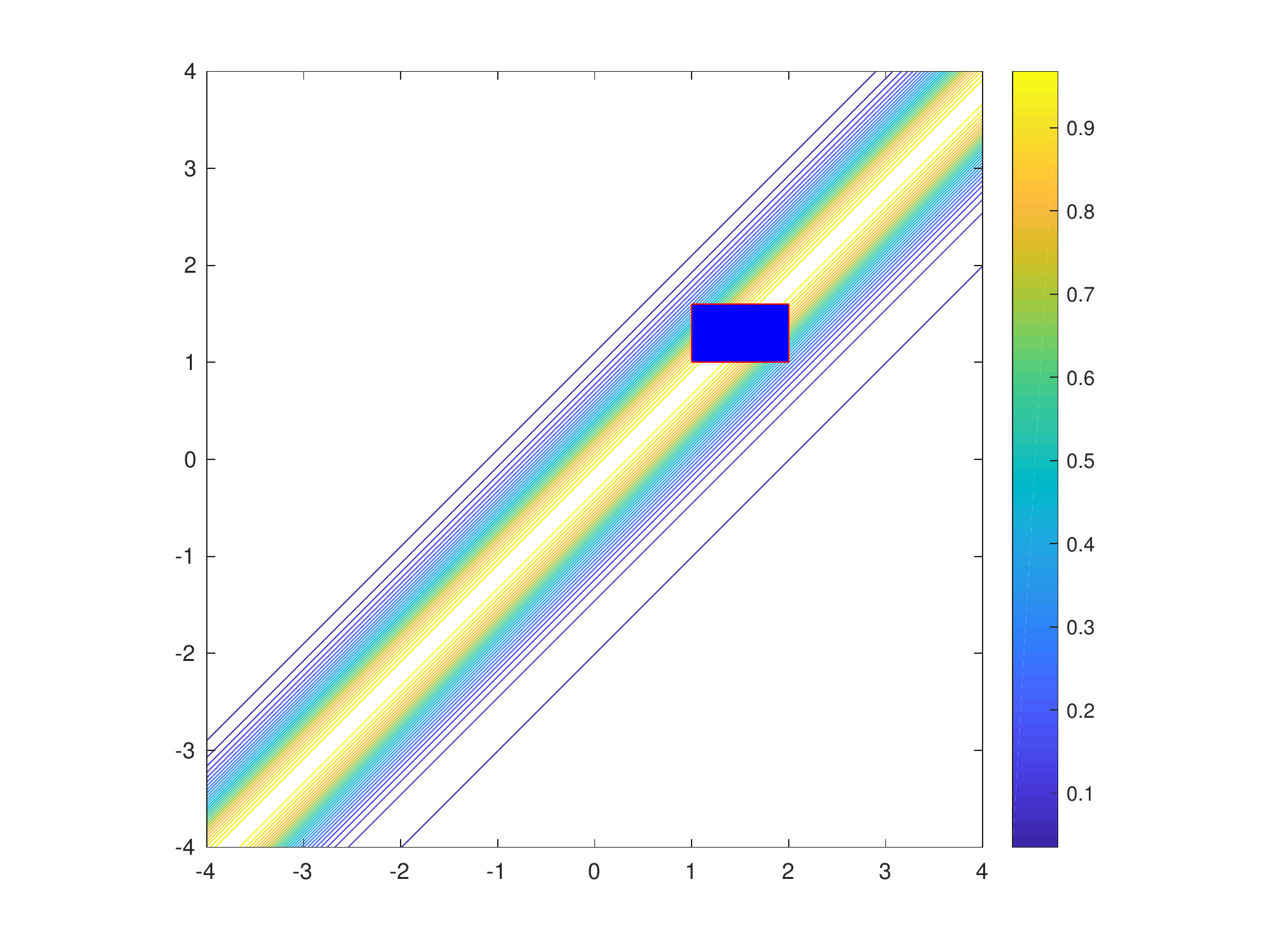}}&
\resizebox{0.5\textwidth}{!}{\includegraphics{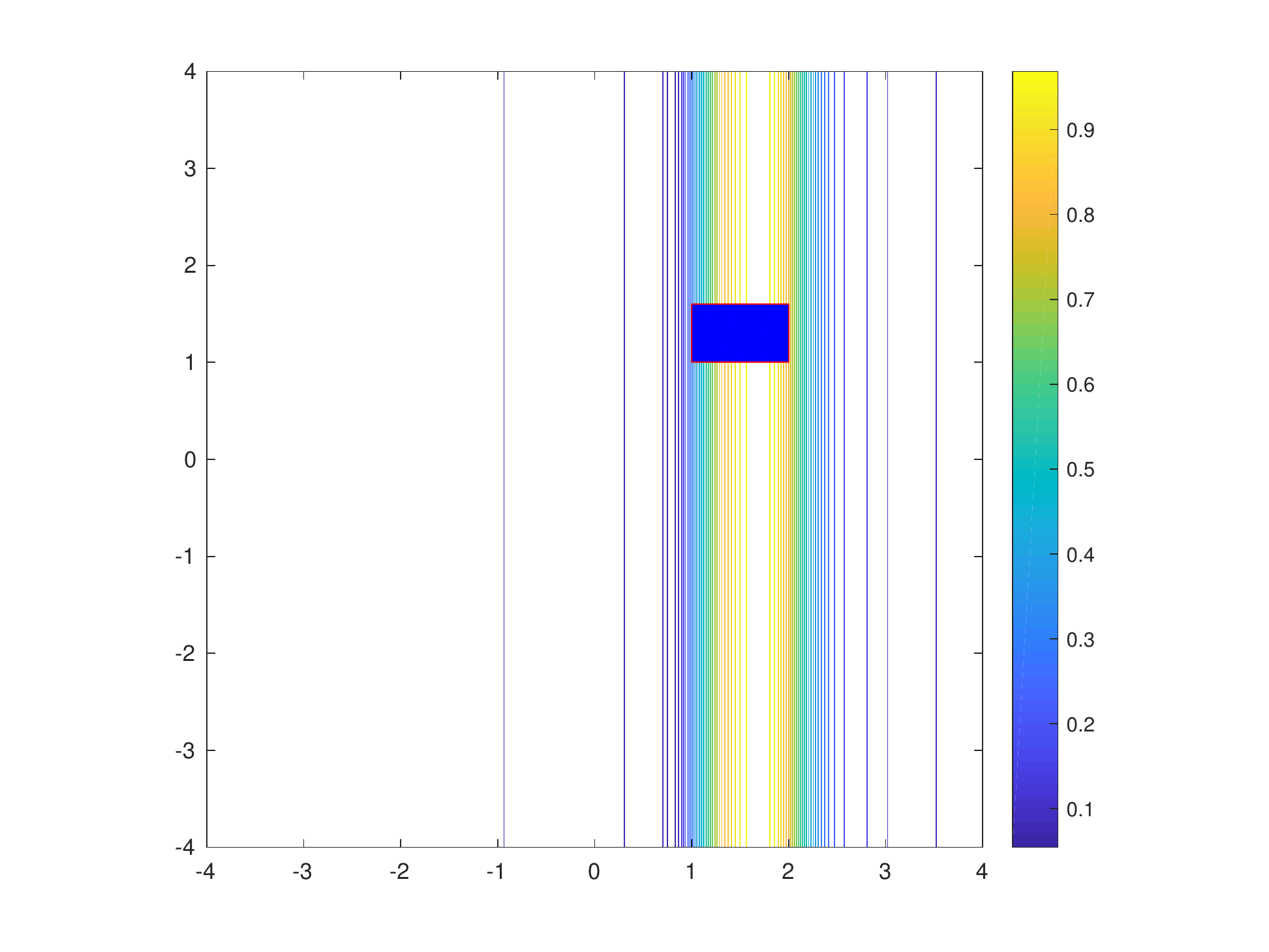}}\\
\resizebox{0.5\textwidth}{!}{\includegraphics{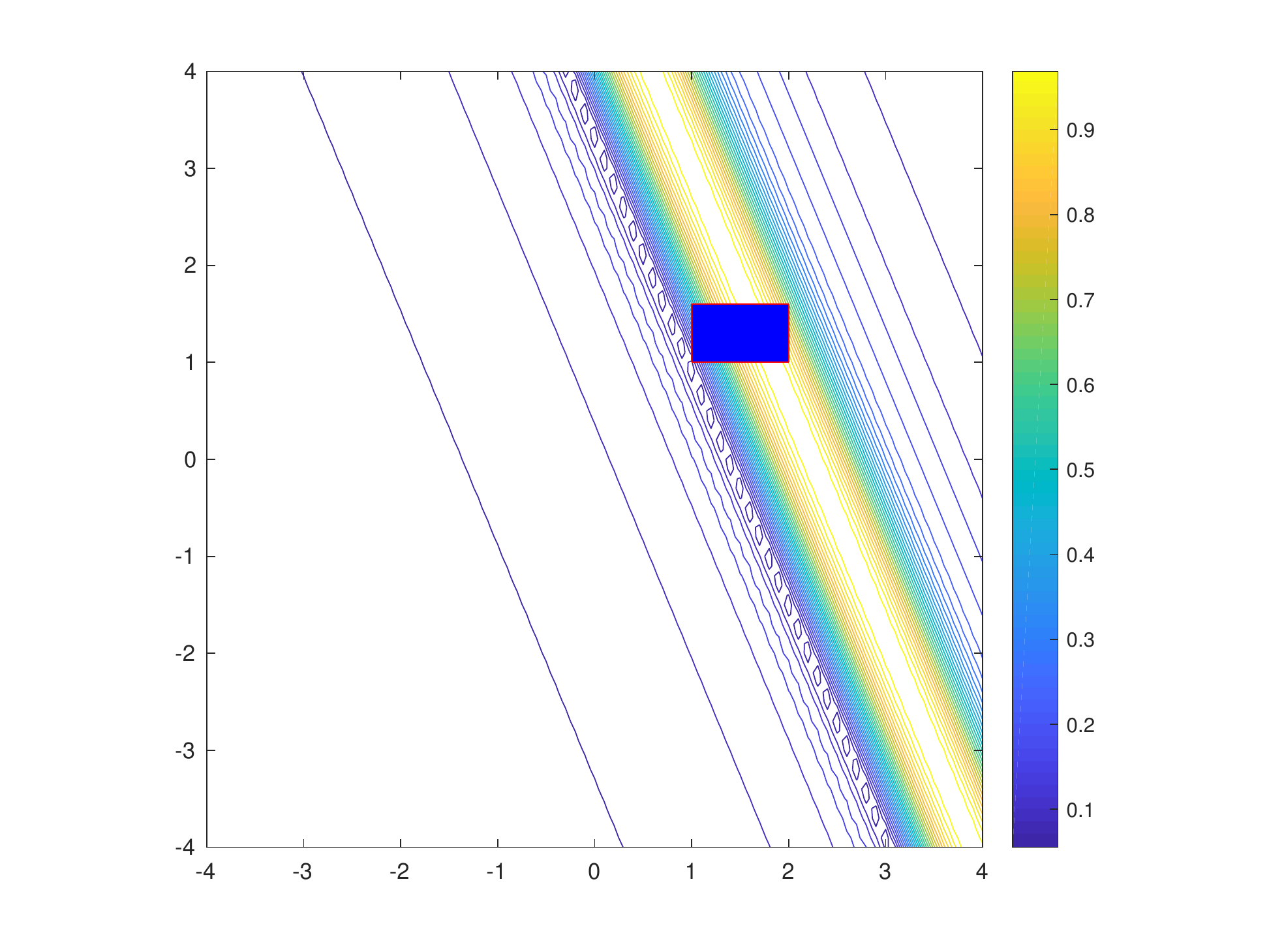}}&
\resizebox{0.5\textwidth}{!}{\includegraphics{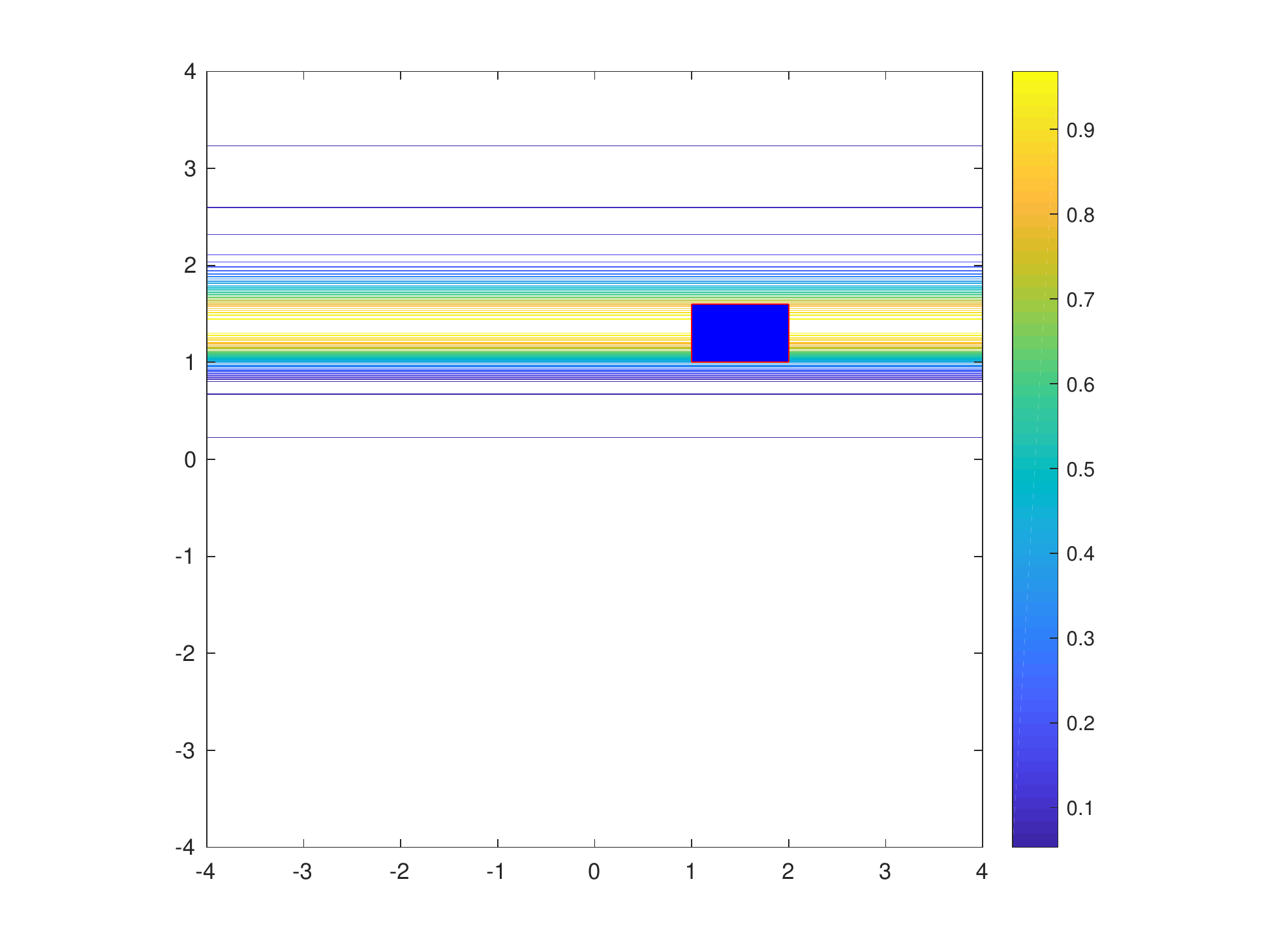}}
\end{tabular}
\end{center}
\caption{Reconstructions using a single observation direction for one object when $F=5$. Top Left: $\varphi = -\pi/4$. Top Right: $\varphi = 0$. Bottom Left: $\varphi = \pi/8$. Bottom Right: $\varphi = \pi/2$.}
\label{OneOneF1}
\end{figure}

In Fig.~\ref{TwoOneF1}, we show the results when the support of the source has two components. One is a rectangle given by $(1, 1.6) \times (1, 1.4)$. The other one is a disc with
radius $0.2$ centered at $(-0.5, -0.5)$. For different observation directions, strips containing the objects are reconstructed effectively.
\begin{figure}
\begin{center}
\begin{tabular}{cc}
\resizebox{0.5\textwidth}{!}{\includegraphics{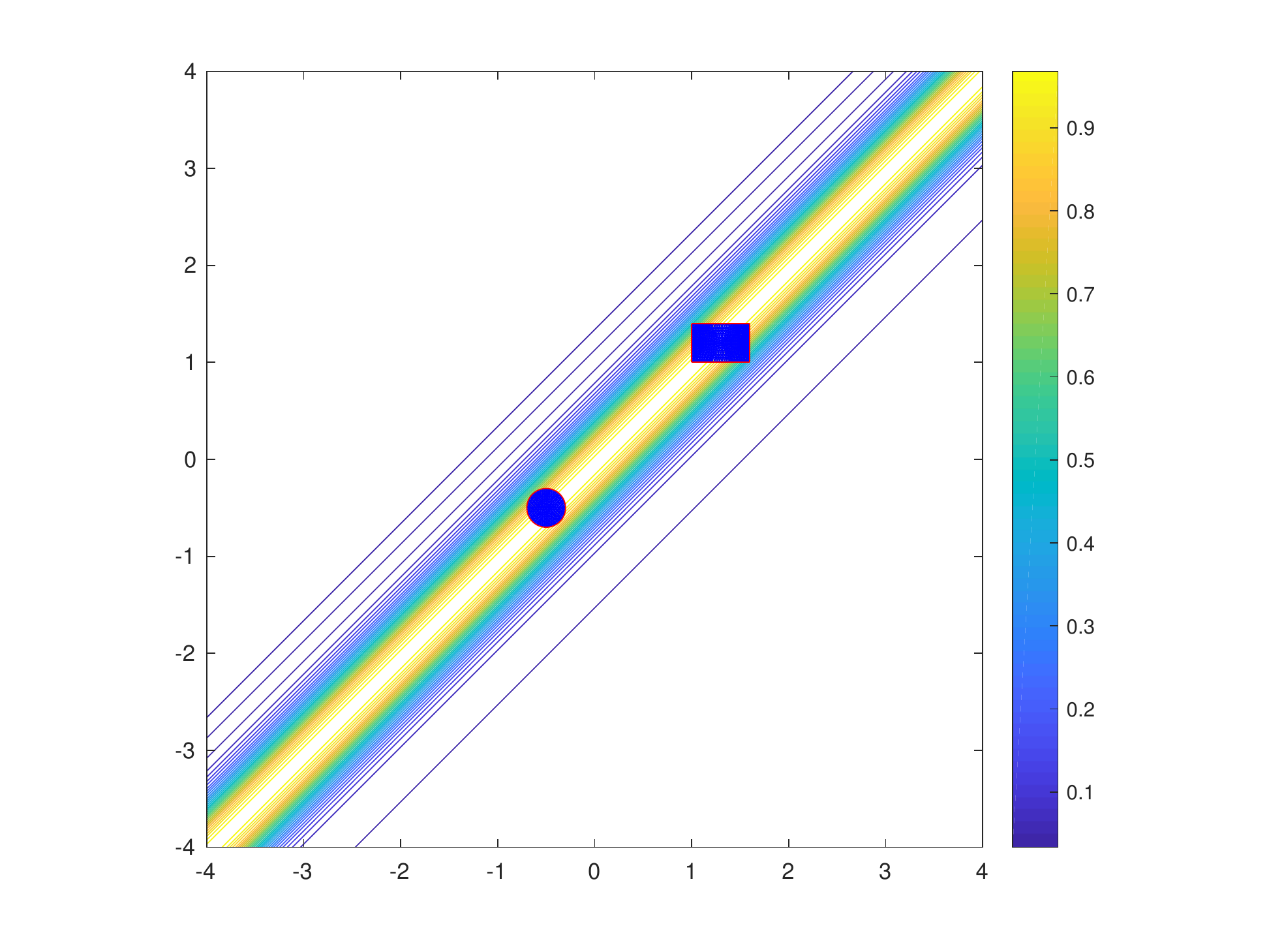}}&
\resizebox{0.5\textwidth}{!}{\includegraphics{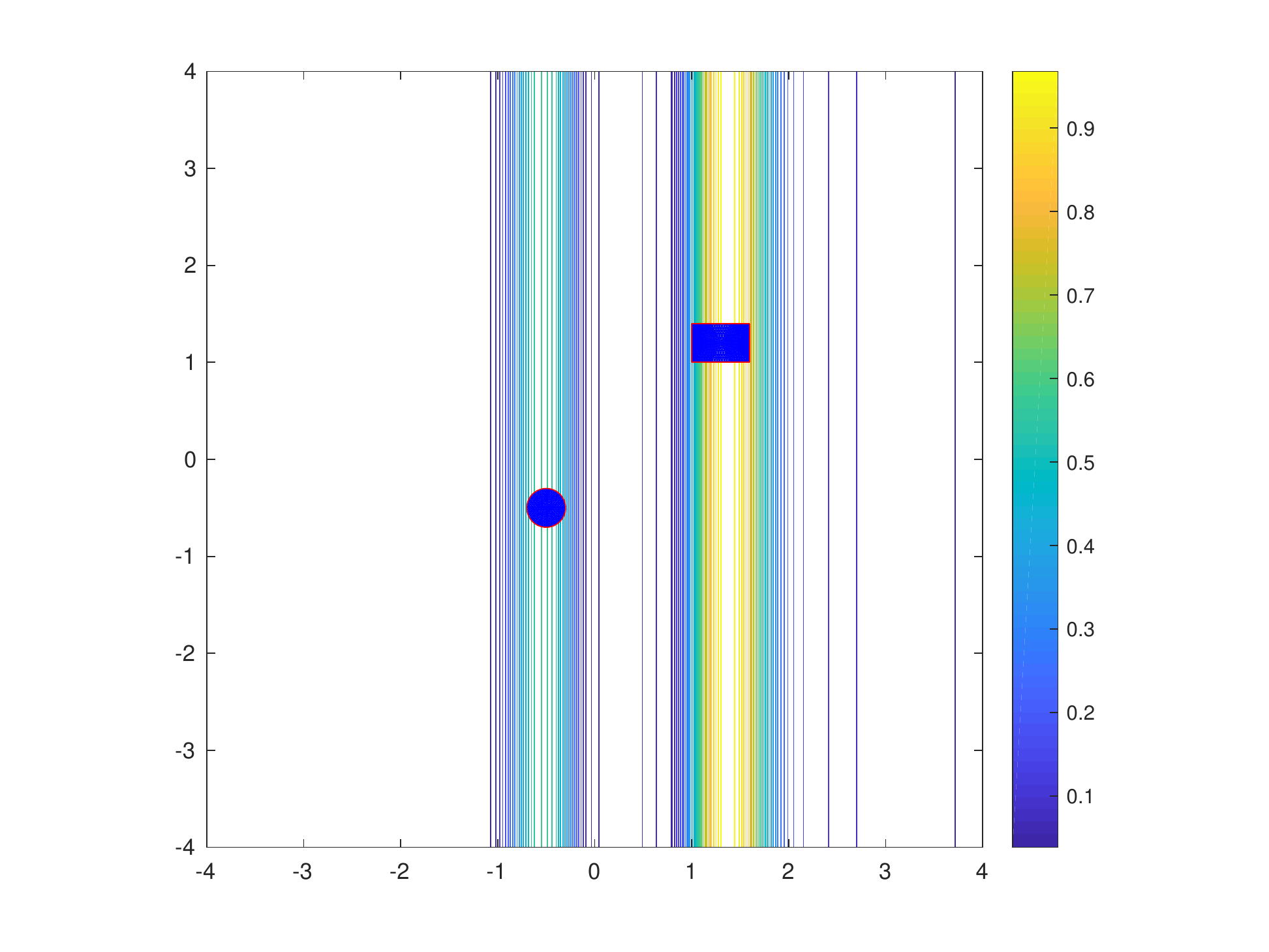}}\\
\resizebox{0.5\textwidth}{!}{\includegraphics{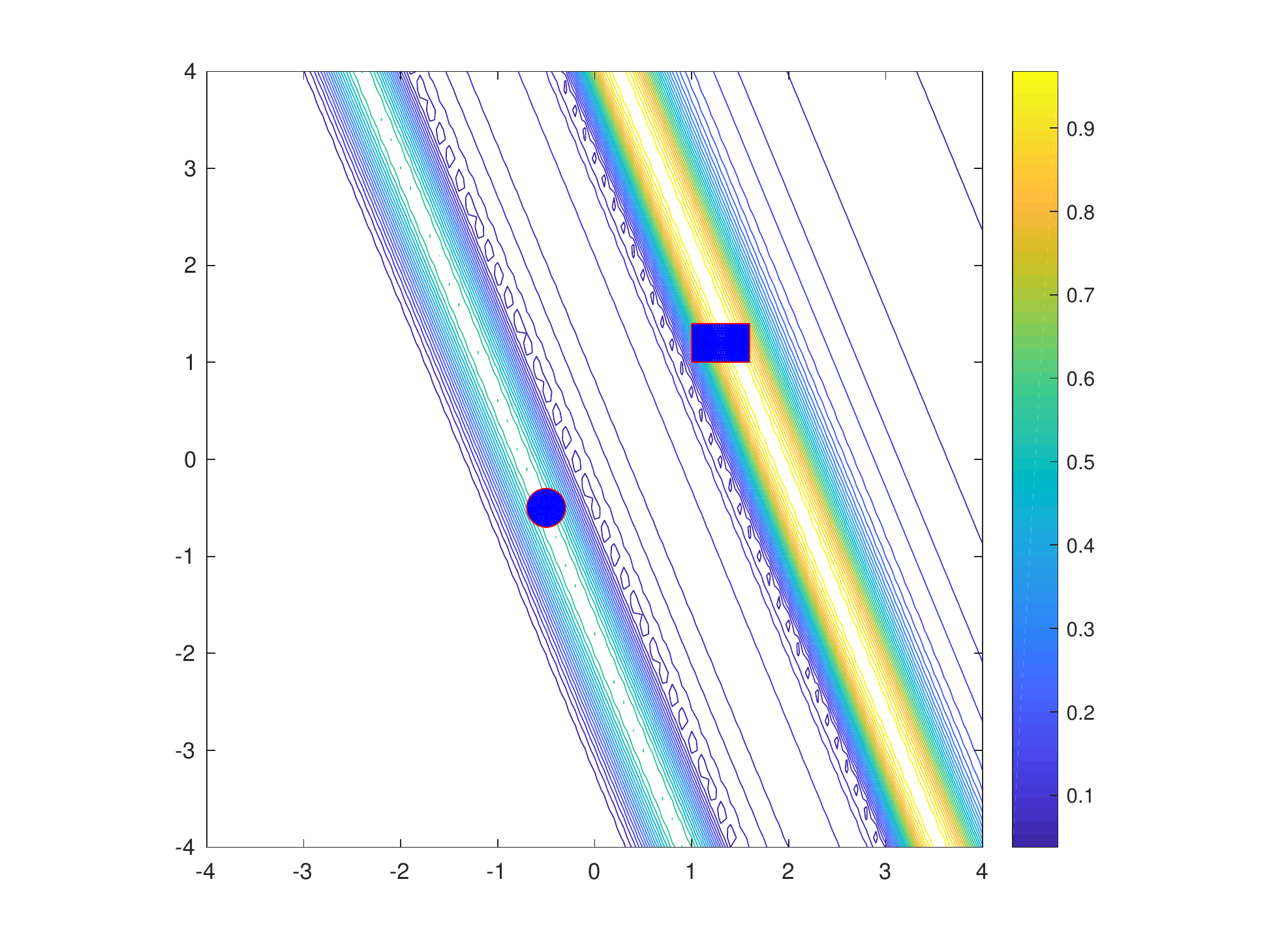}}&
\resizebox{0.5\textwidth}{!}{\includegraphics{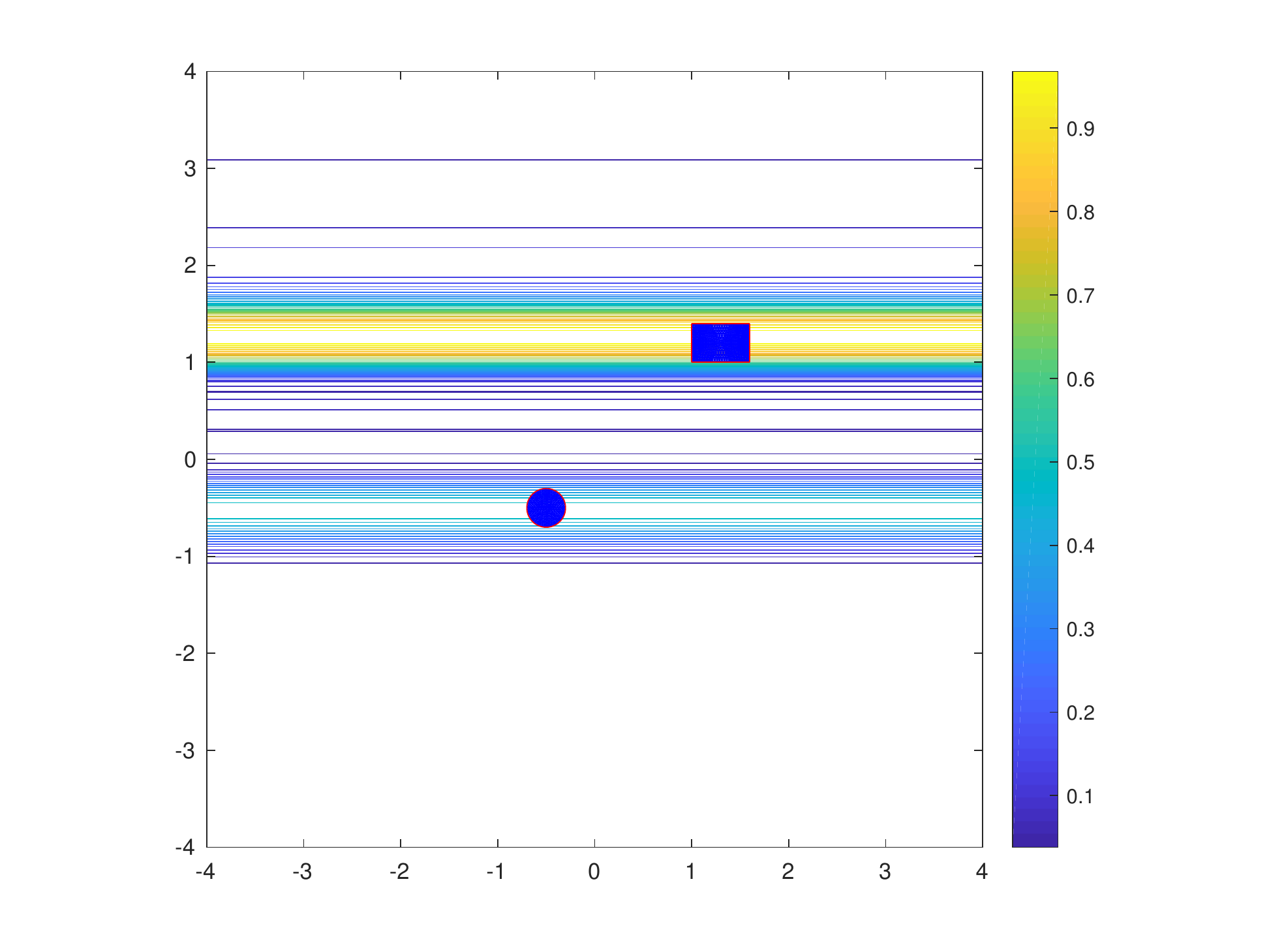}}
\end{tabular}
\end{center}
\caption{Reconstructions using a single observation direction for two objects when $F=5$. Top Left: $\varphi = -\pi/4$. Top Right: $\varphi = 0$. Bottom Left: $\varphi = \pi/8$. Bottom Right: $\varphi = \pi/2$.}
\label{TwoOneF1}
\end{figure}

\subsection{Two observation directions}
Now we consider two observation angles: $\varphi_1 = 0$ and $\varphi_2 = \pi/2$. We compute the indicators and superimpose them in one picture.
Since the observation directions are perpendicular to each other, the strips are perpendicular to each other in Fig.~\ref{TwoOneTwoF1}.  For both one object and
two objects, we see that intersection of the strips contains the support of $F$.
\begin{figure}
\begin{center}
\begin{tabular}{cc}
\resizebox{0.5\textwidth}{!}{\includegraphics{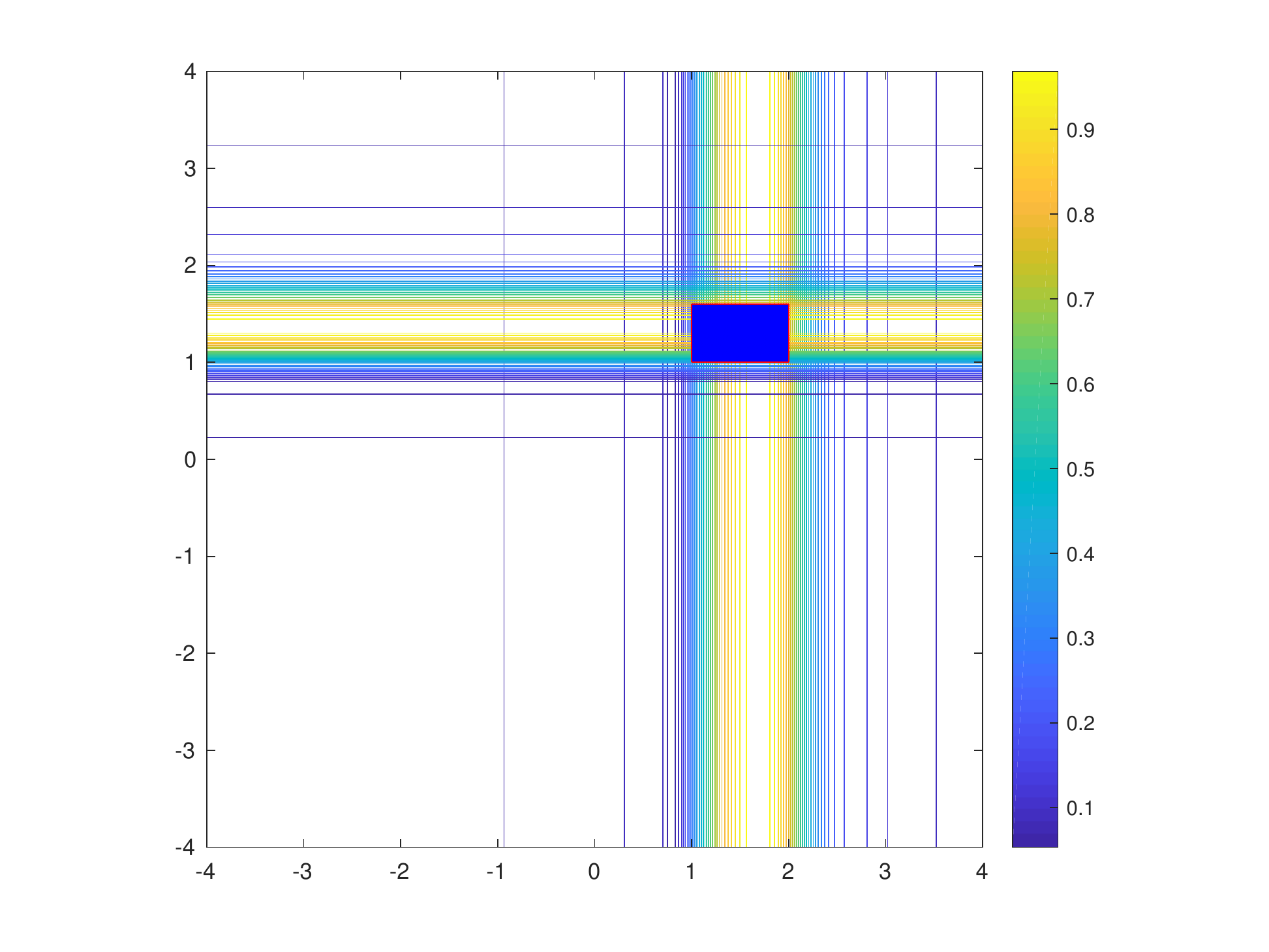}}&
\resizebox{0.5\textwidth}{!}{\includegraphics{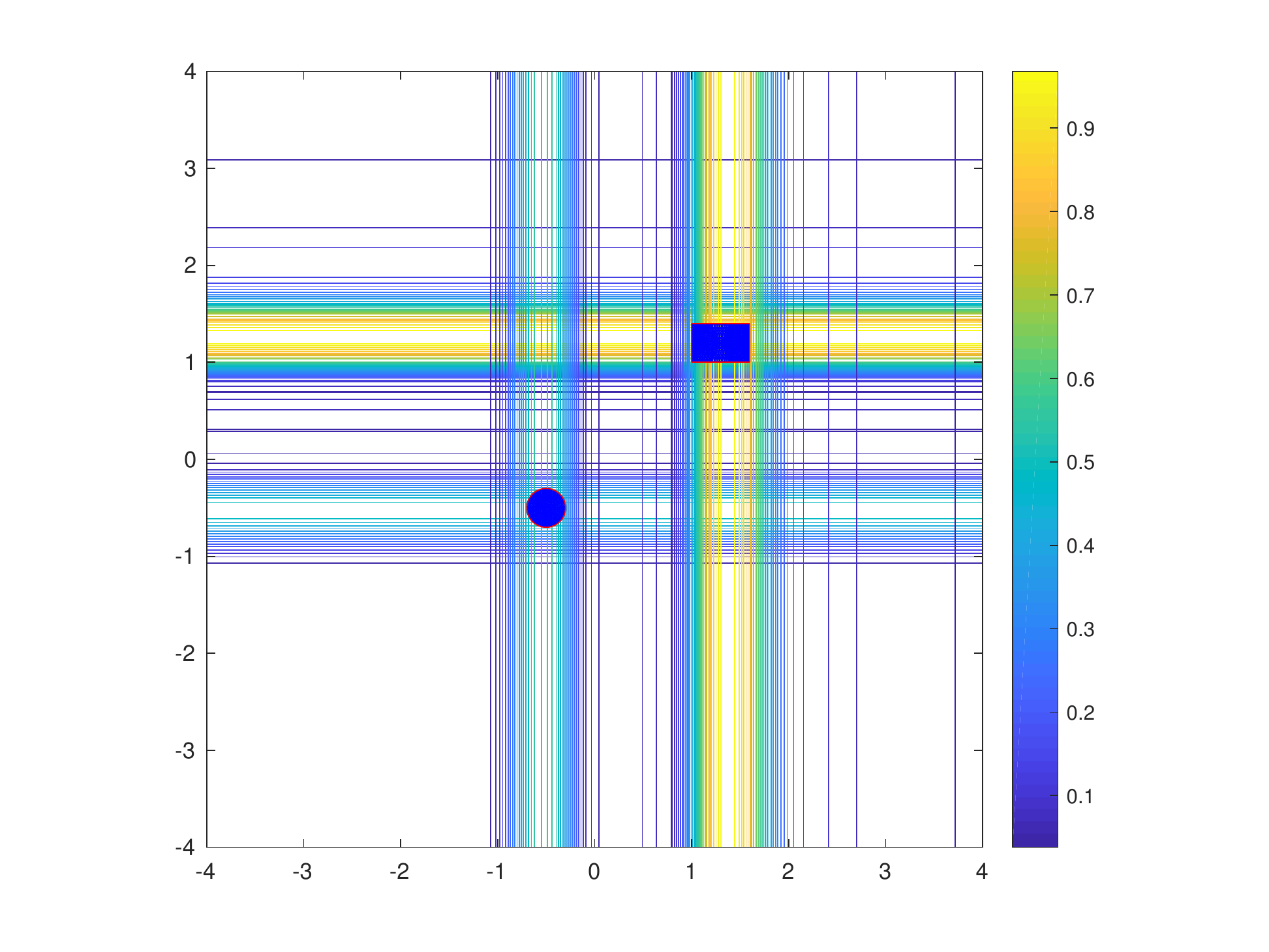}}
\end{tabular}
\end{center}
\caption{Reconstructions using two observation angles $\varphi = \pi/2$ and $\varphi = 0$ when $F=5$. Left: Single object; Right: Two objects.}
\label{TwoOneTwoF1}
\end{figure}

\subsection{Multiple observation directions}
Now we use $M=20$ observation angles $\varphi_j, j=1,\ldots, 20$ such that $\varphi_j = -\pi/2+j\pi/M$.
Note that $\varphi_j\in (-\pi/2, \pi/2]$.
We superimpose the indicators and show the results in Fig.~\ref{OneFullF1}.
The locations and sizes of support of $F$ are reconstructed correctly.
 \begin{figure}
\begin{center}
\begin{tabular}{cc}
\resizebox{0.5\textwidth}{!}{\includegraphics{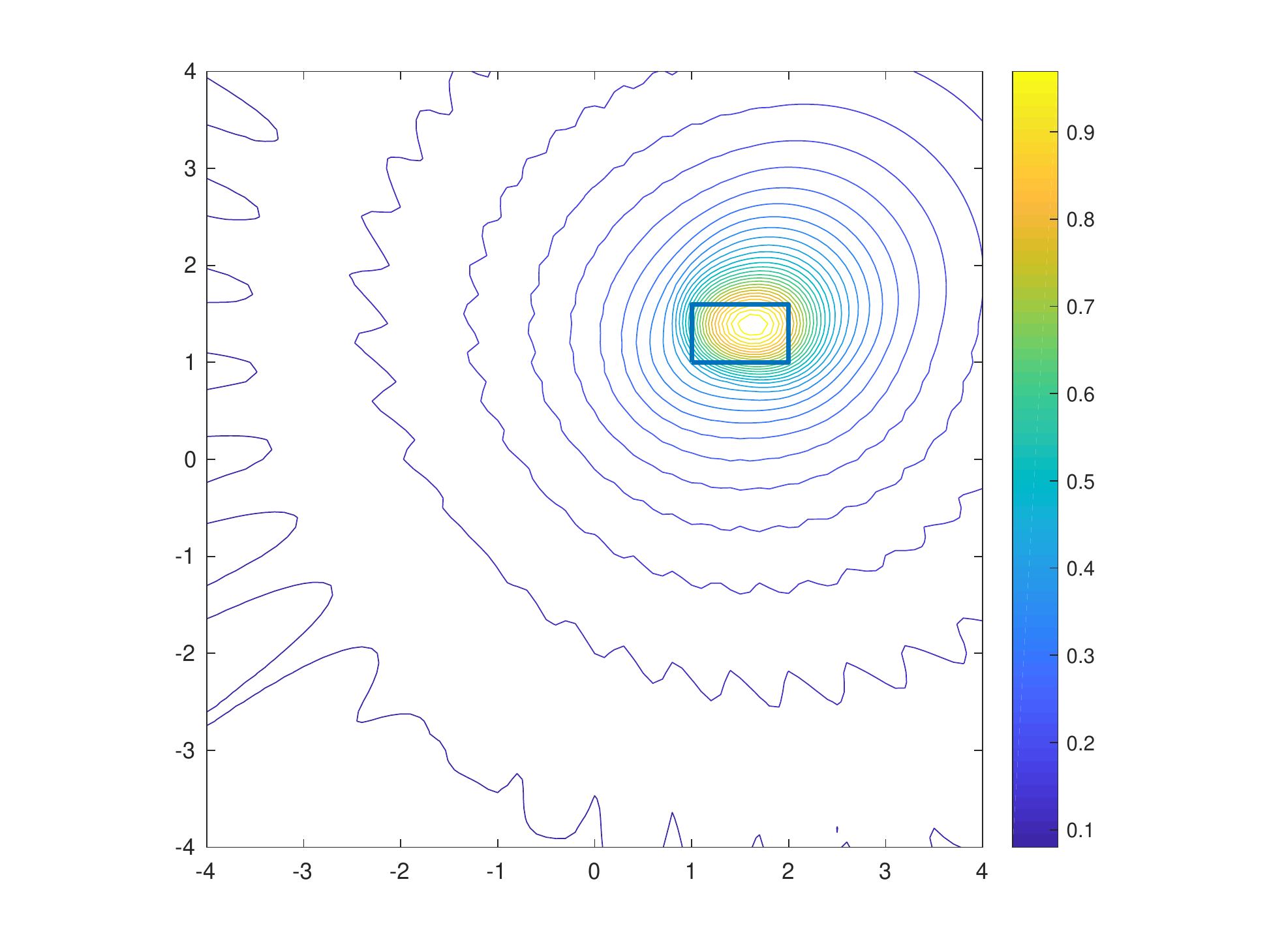}}&
\resizebox{0.5\textwidth}{!}{\includegraphics{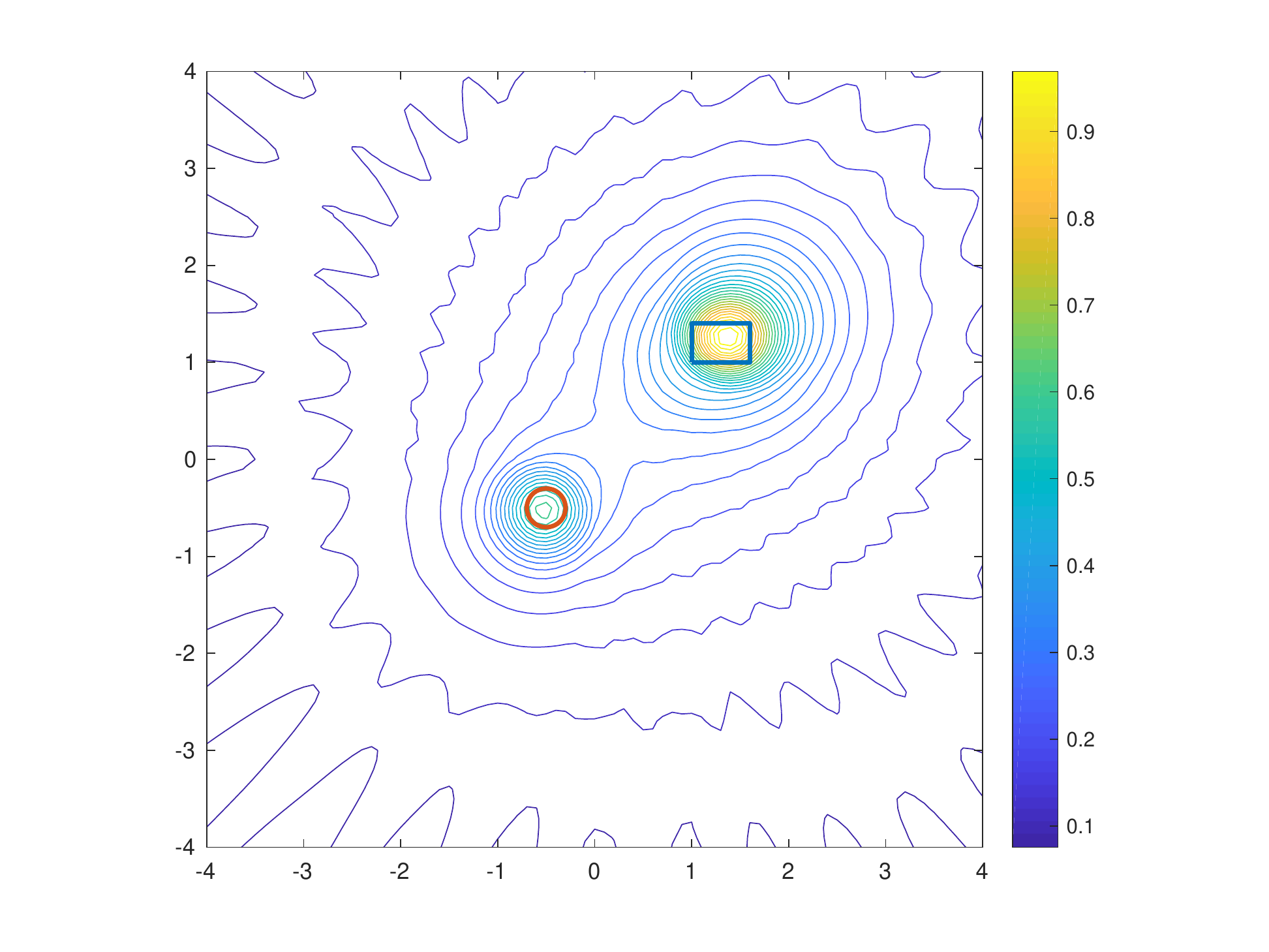}}
\end{tabular}
\end{center}
\caption{Reconstruction using multiple observation directions when $F=5$. Left: single object. Right: Two objects.}
\label{OneFullF1}
\end{figure}

Next, we choose $F(x,y) = x^2-y^2+5$, a function depending on the locations but independent of the wave number $k$.
The reconstruction is shown in Fig.~\ref{OneFullF2}.
\begin{figure}
\begin{center}
\begin{tabular}{cc}
\resizebox{0.5\textwidth}{!}{\includegraphics{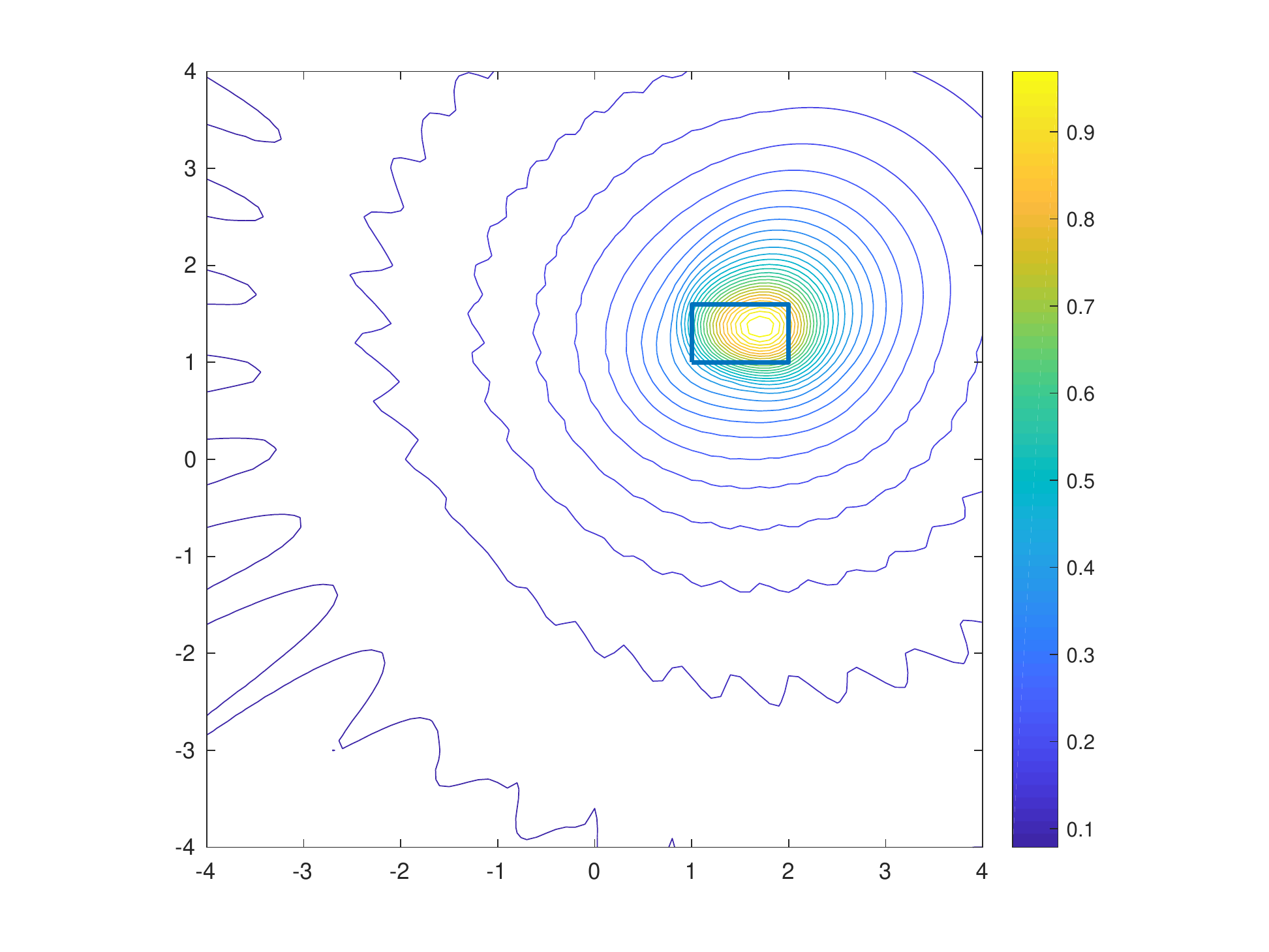}}&
\resizebox{0.5\textwidth}{!}{\includegraphics{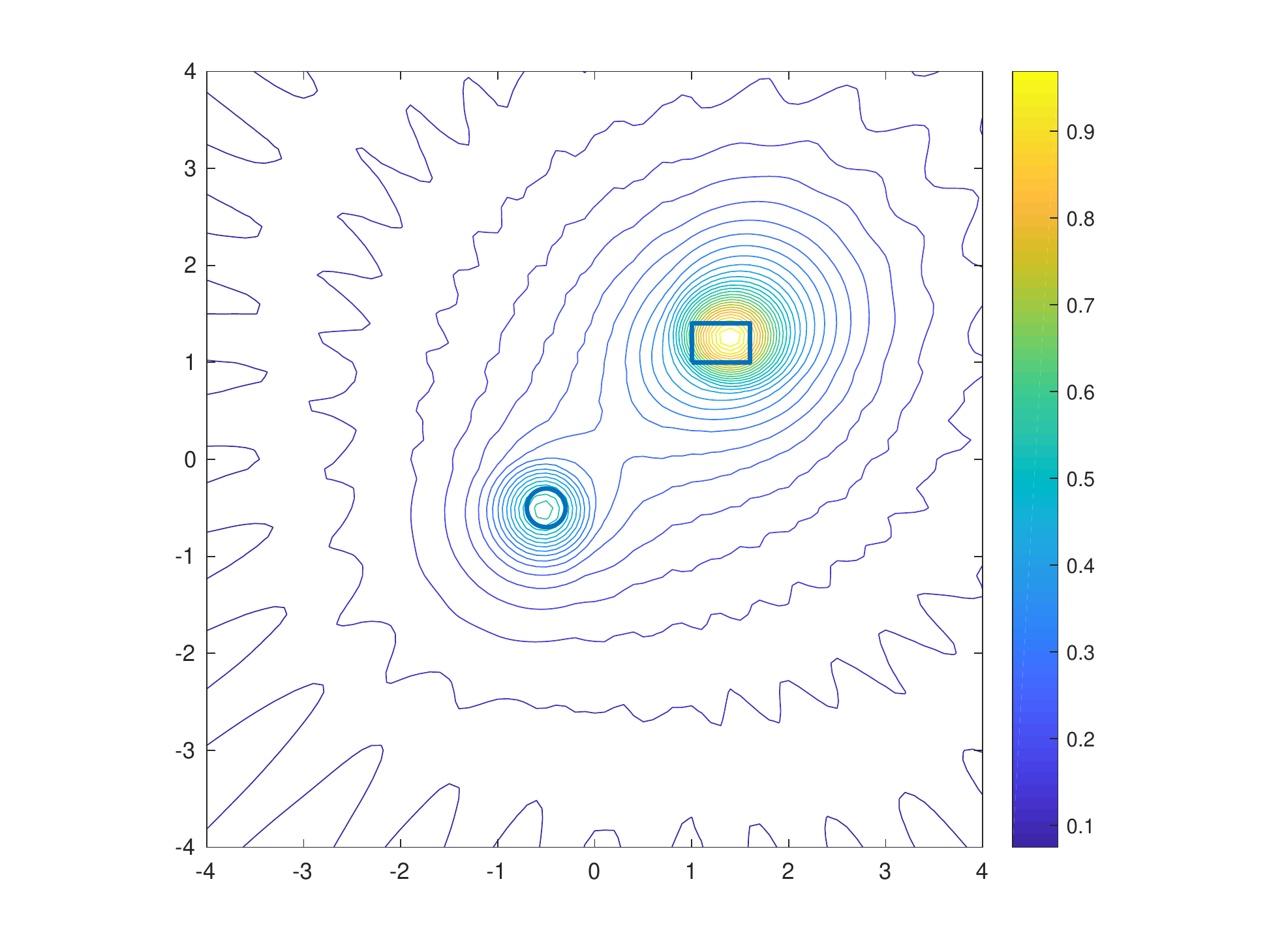}}
\end{tabular}
\end{center}
\caption{Reconstructions using multiple observation directions when $F(x,y) = x^2-y^2+5$. Left: single object. Right: Two objects.}
\label{OneFullF2}
\end{figure}

We also consider the case when the source $F$ depends on $k$ as well. Let
\[
F_1(x, y; k) = k^2(x^2-y^2+5),
\]
and
\[
F_2(x, y; k) = e^{ik(x\cos 3\pi/2 + y \sin 3\pi/2)}(x^2-y^2+5).
\]
The reconstructions are shown in Fig.~\ref{OneFullF3}.
Combining the previous Figures \ref{OneFullF1}-\ref{OneFullF2}, we observe that the reconstructions change slightly for different source function,
and the location and size of the support are always well captured.

\begin{figure}
\begin{center}
\begin{tabular}{cc}
\resizebox{0.5\textwidth}{!}{\includegraphics{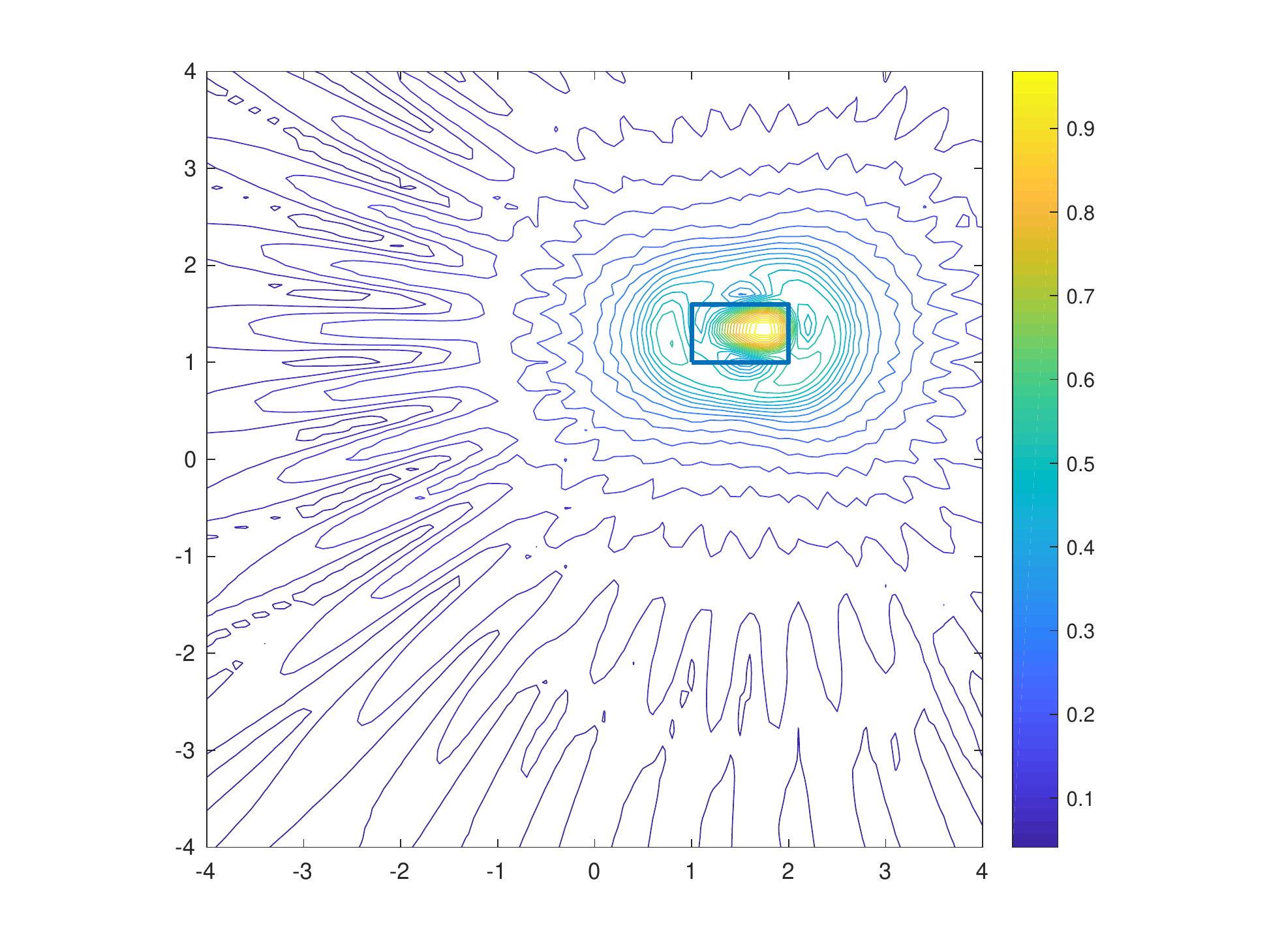}}&
\resizebox{0.5\textwidth}{!}{\includegraphics{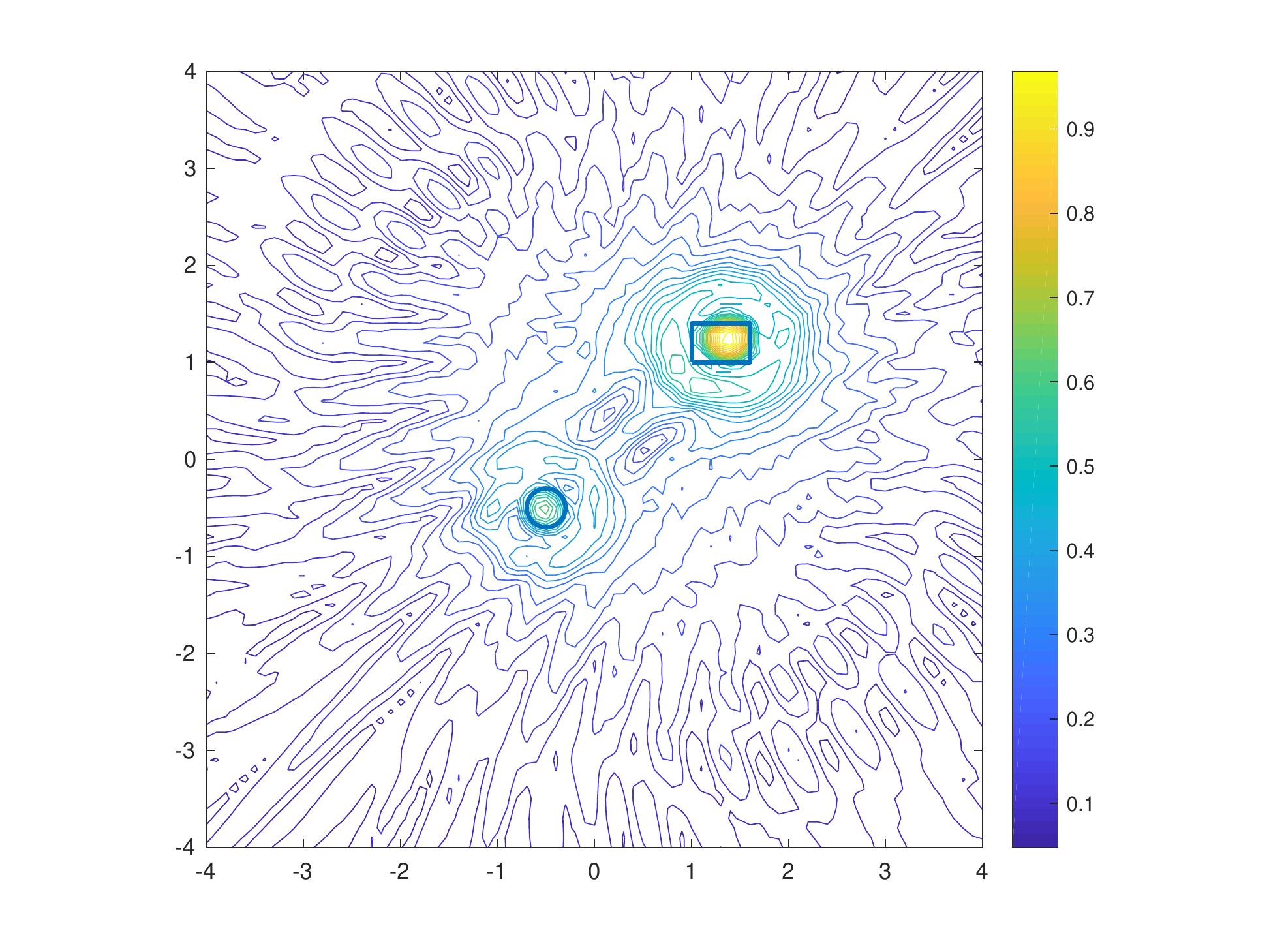}}\\
\resizebox{0.5\textwidth}{!}{\includegraphics{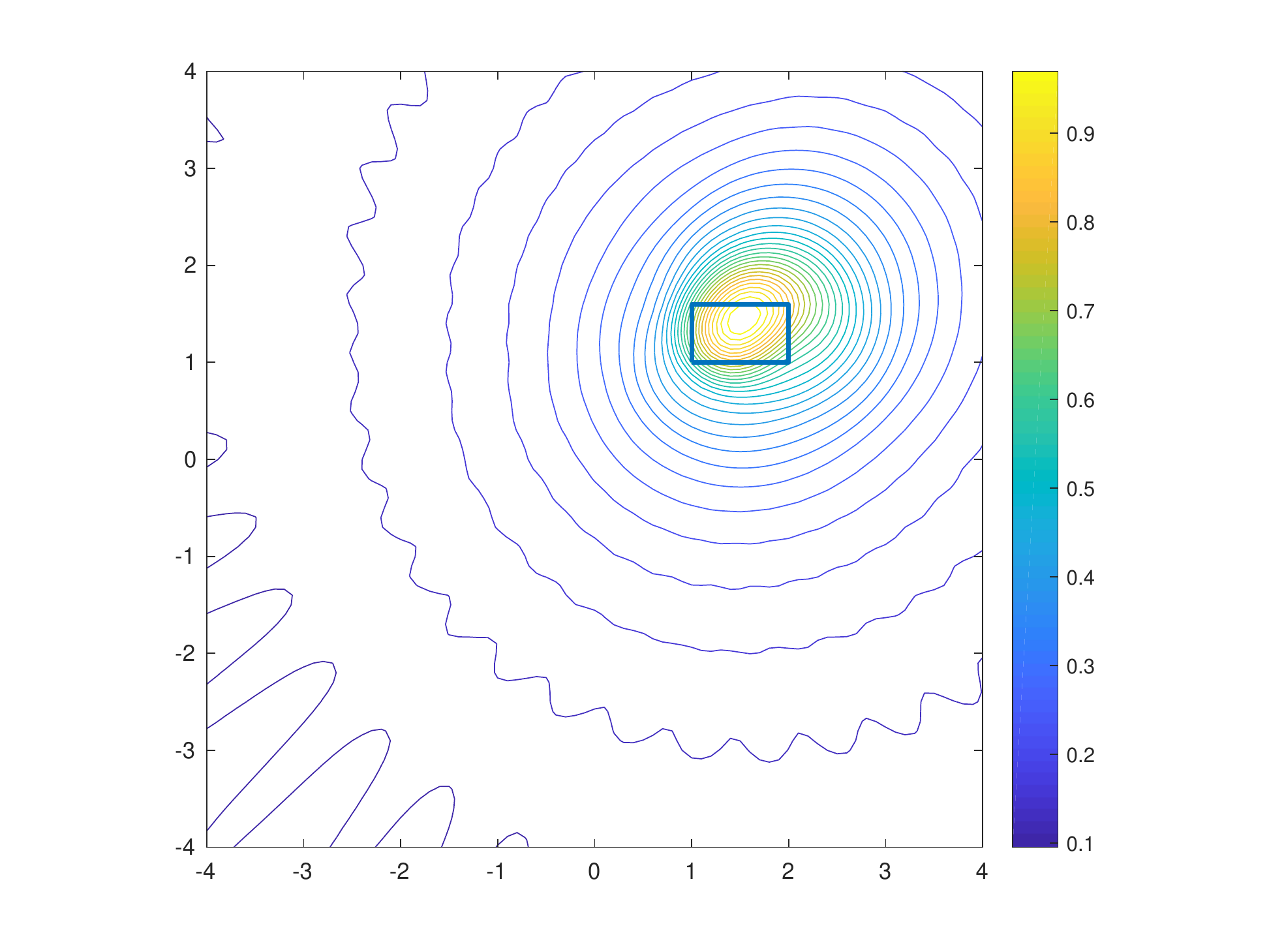}}&
\resizebox{0.5\textwidth}{!}{\includegraphics{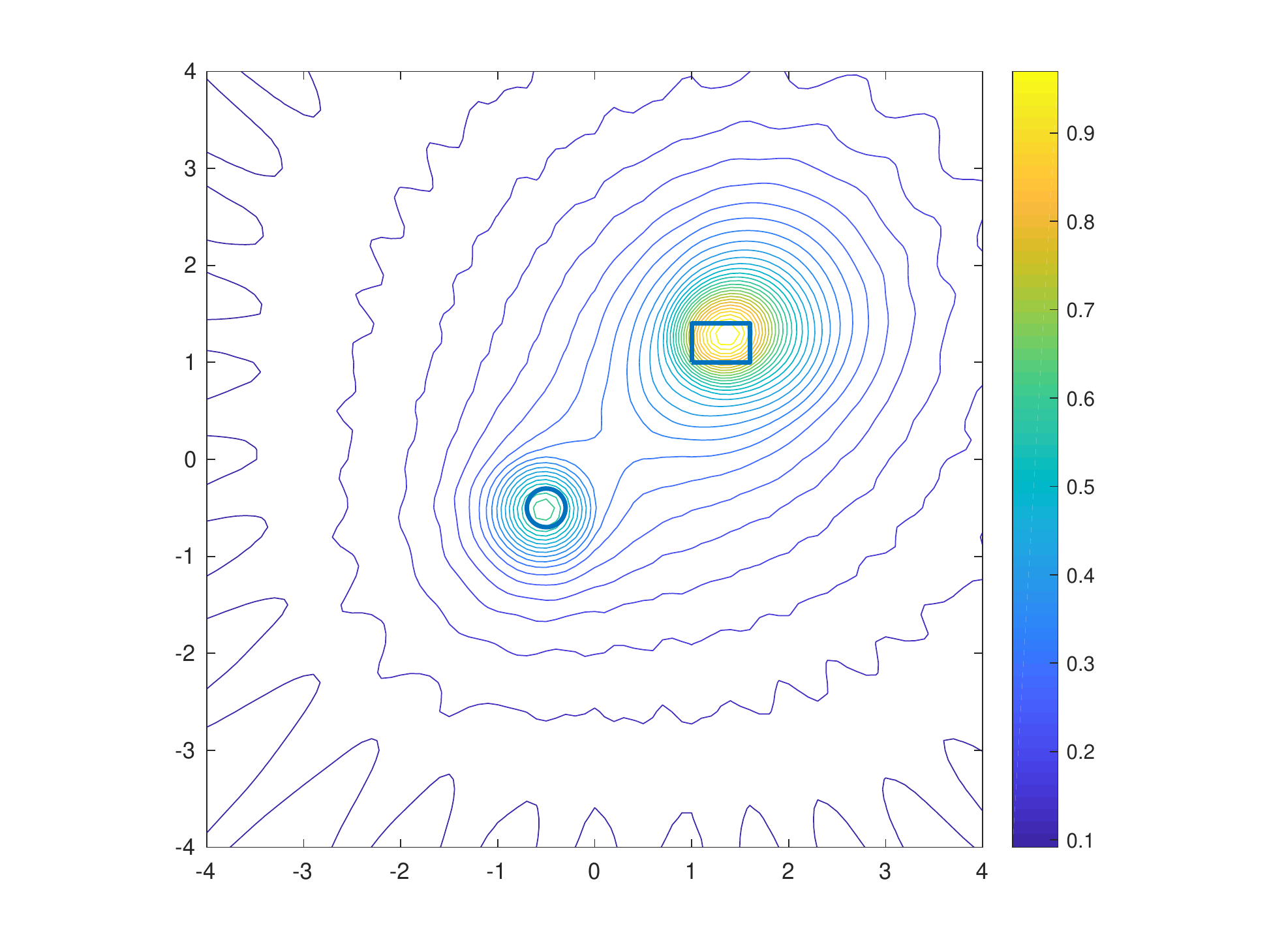}}
\end{tabular}
\end{center}
\caption{Reconstructions of sources depending on wave number $k$. Top:  $F_1(x, y; k) = k^2(x^2-y^2+5)$. Top Left: one object. Top Right: two objects.
Bottom: $F_2(x, y; k) = e^{ik(x\cos  3\pi/2+ y \sin 3\pi/2)}(x^2-y^2+5)$. Bottom Left: one object. Bottom Right: two objects.}
\label{OneFullF3}
\end{figure}

\subsection{Extended objects}
The sizes of supports of $F$ in the above examples are small compared with the wavelengths used.
The smallest wavelength is $\lambda_{min} = 2\pi/10 \approx 0.628$.

In Figures \ref{LargeOneF1}-\ref{LargeOneF2}, we show the reconstructions of larger objects.
Figure \ref{LargeOneF1} shows that the reconstructions of the source given in \eqref{f1} with a single observation direction.
This example further shows that Theorem \ref{theorem1} does not hold in general if the set in \eqref{set1} has
positive Lebesgue measure.  In fact, for observation direction $(0,1)$, a gap appears clearly in $(-1,1)$.

\begin{figure}
\begin{center}
\begin{tabular}{cc}
\resizebox{0.5\textwidth}{!}{\includegraphics{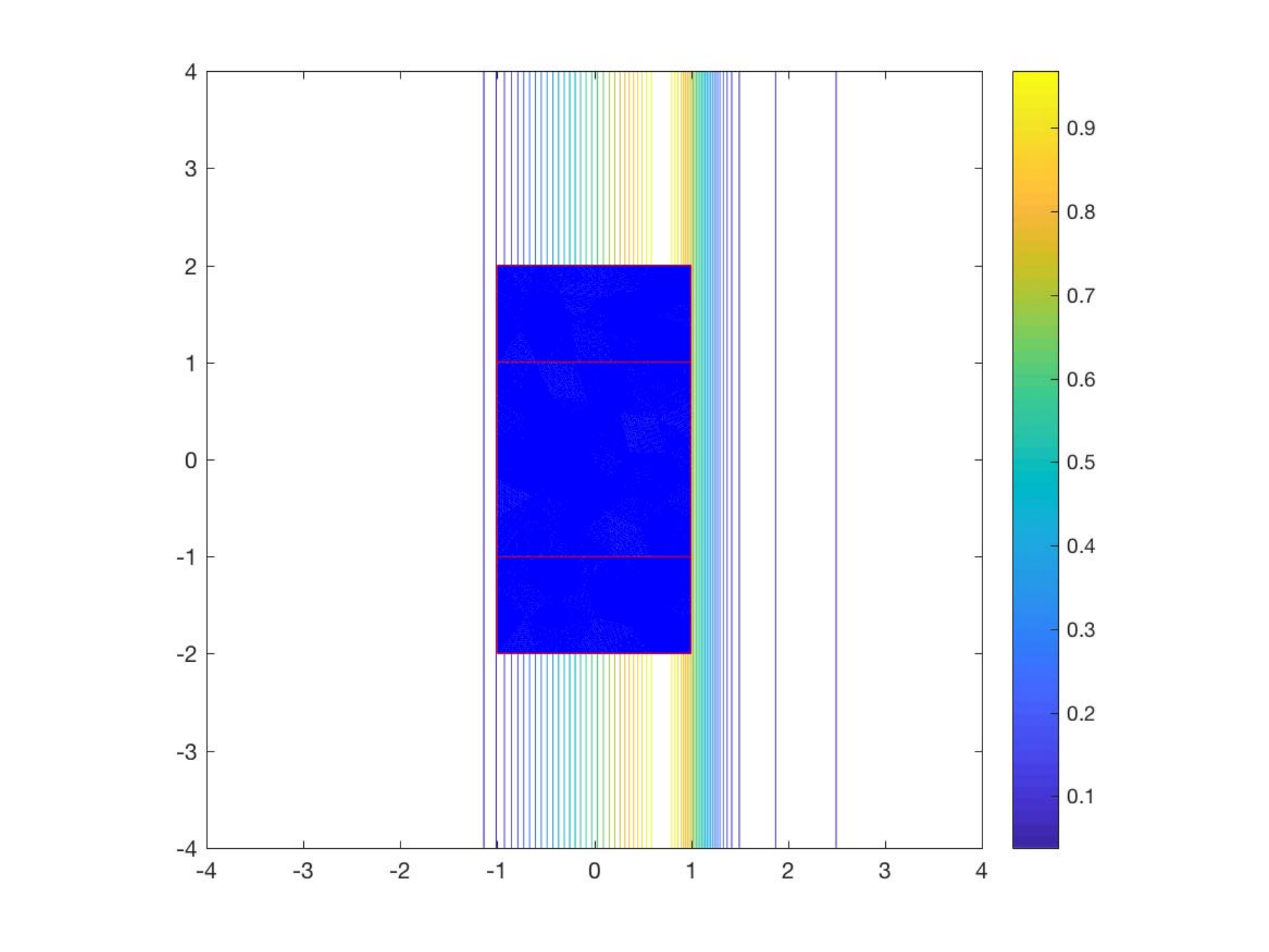}}&
\resizebox{0.5\textwidth}{!}{\includegraphics{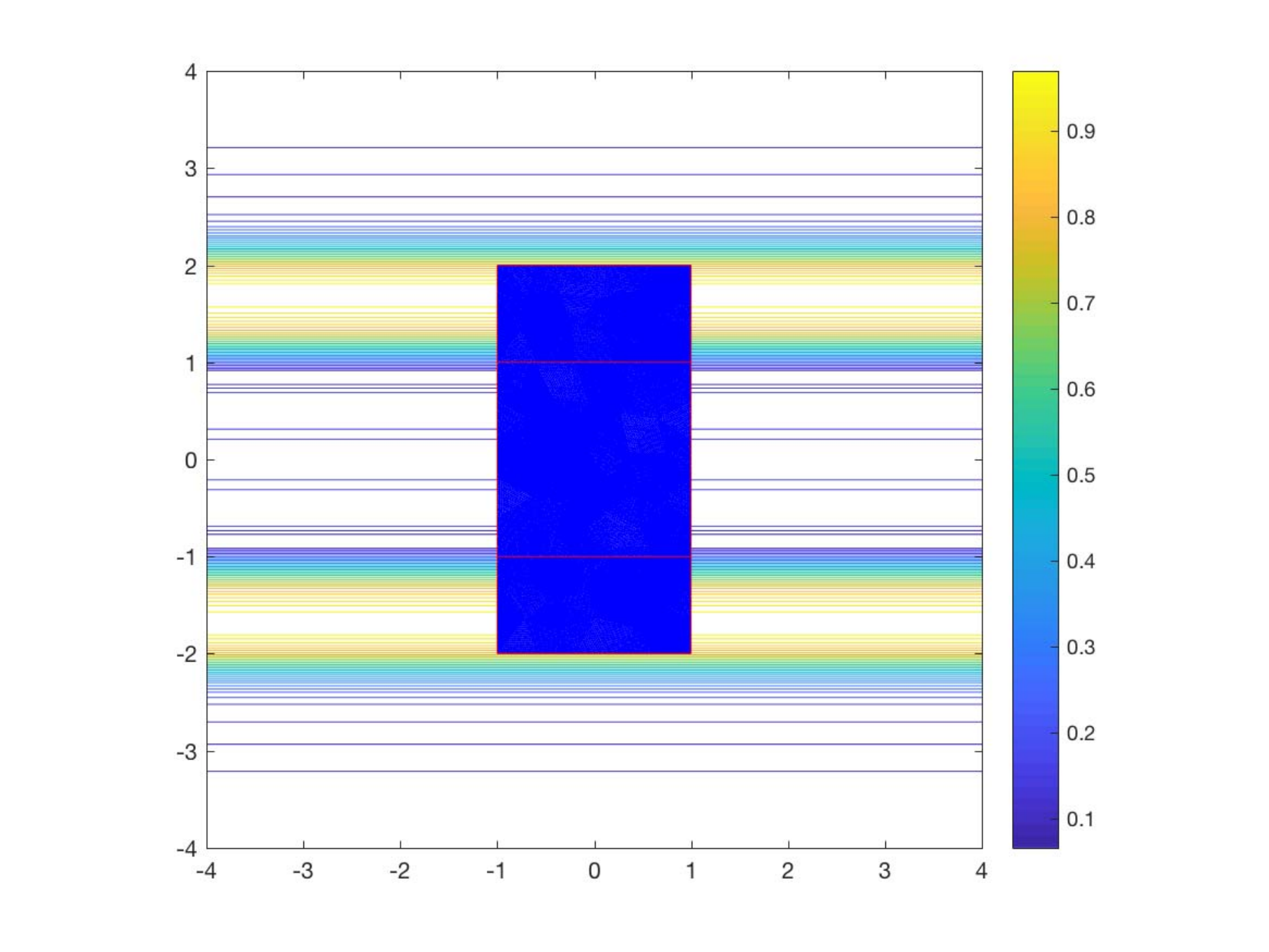}}\\
\resizebox{0.5\textwidth}{!}{\includegraphics{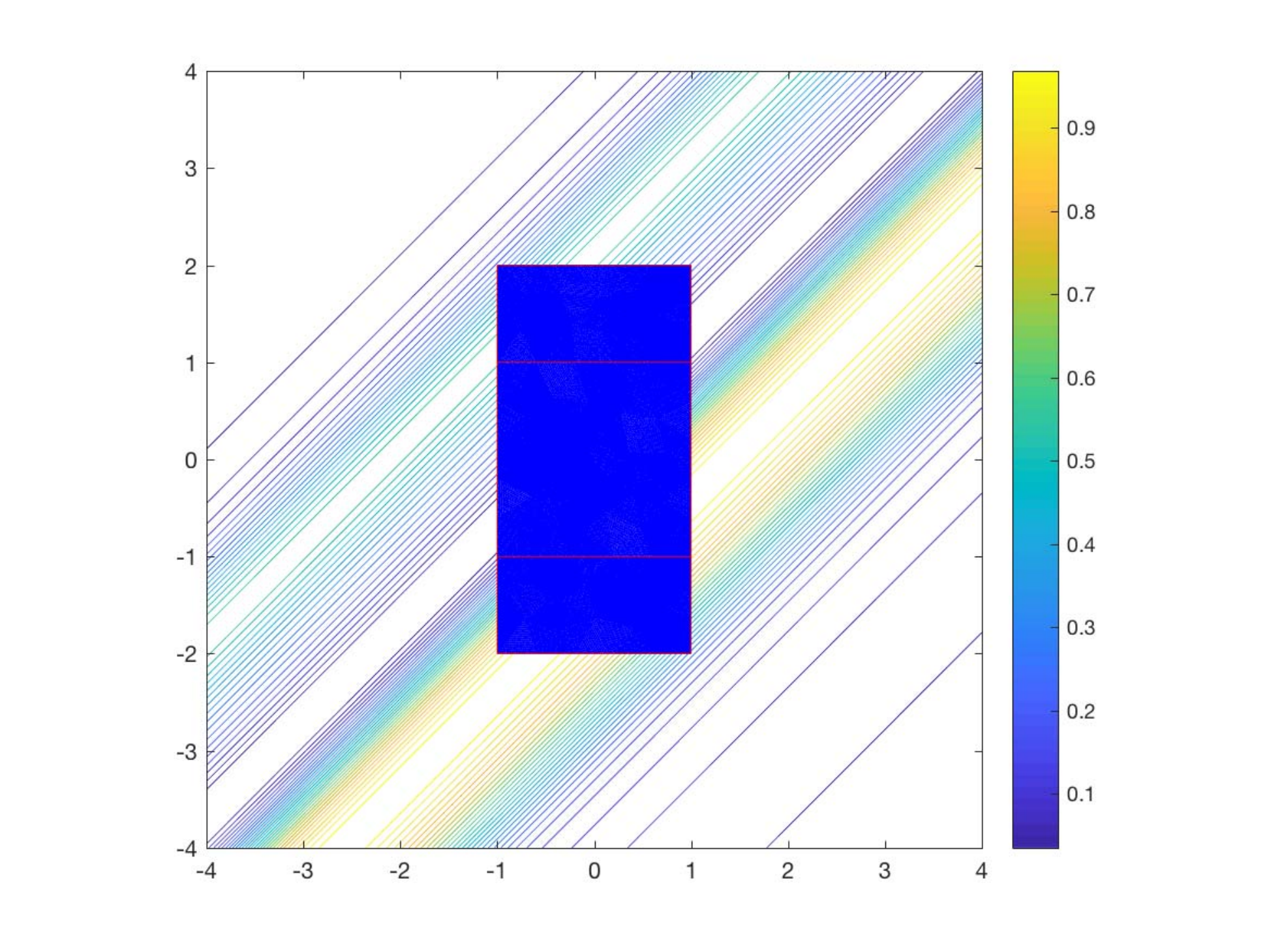}}&
\resizebox{0.5\textwidth}{!}{\includegraphics{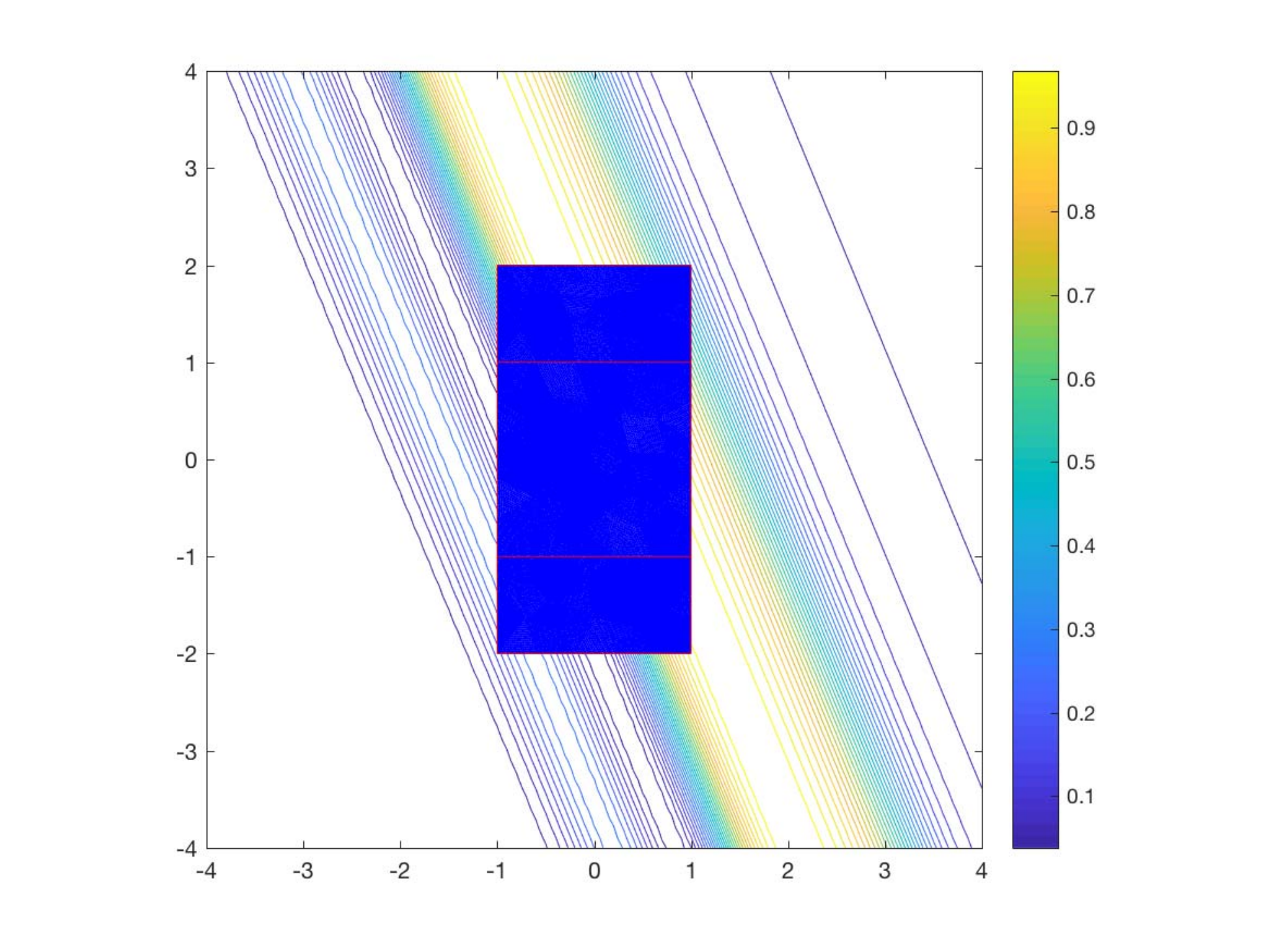}}
\end{tabular}
\end{center}
\caption{Reconstructions of larger object with $F$ given in \eqref{f1} by using a single observation direction.
        Top Left: $\varphi = 0$. Top Right: $\varphi = \pi/2$. Bottom Left: $\varphi = -\pi/4$. Bottom Right: $\varphi = \pi/8$.}
\label{LargeOneF1}
\end{figure}

In Fig.~\ref{LargeOneF2}, we show the reconstructions of larger objects with $M=20$ observation directions.
One is an equilateral triangle with vertices
\[
(-2, 0), \quad (1, 0),  \quad (-1/2, 3/2\sqrt{3}-1).
\]
The second one is a thin slab given by $(-2,2) \times (0, 0.1)$.
The results indicate that shorter wavelength could lead to better reconstruction.
In particular, not only the location and size, but also the shape of the support can be well reconstructed.
\begin{figure}
\begin{center}
\begin{tabular}{cc}
\resizebox{0.5\textwidth}{!}{\includegraphics{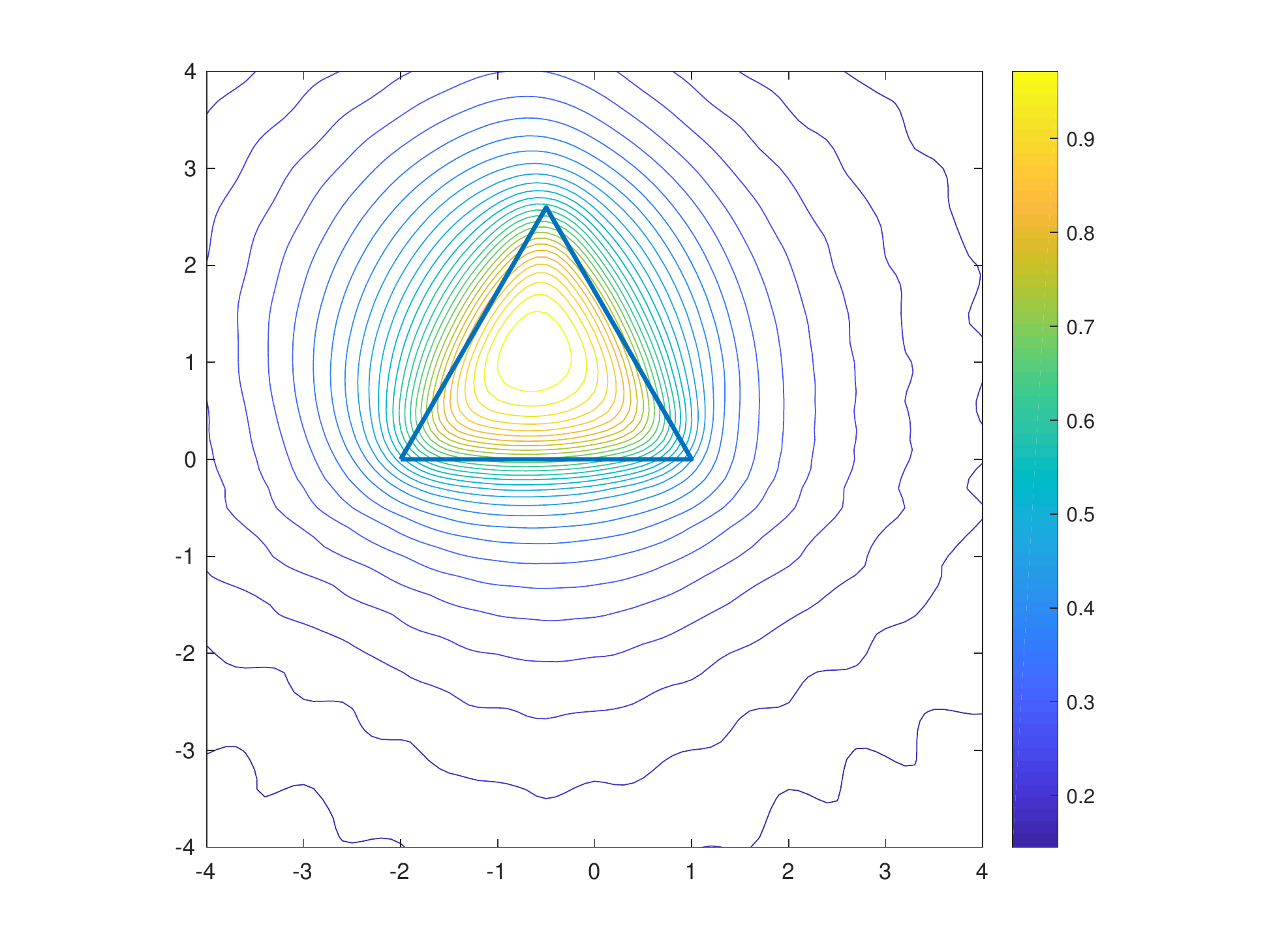}}&
\resizebox{0.5\textwidth}{!}{\includegraphics{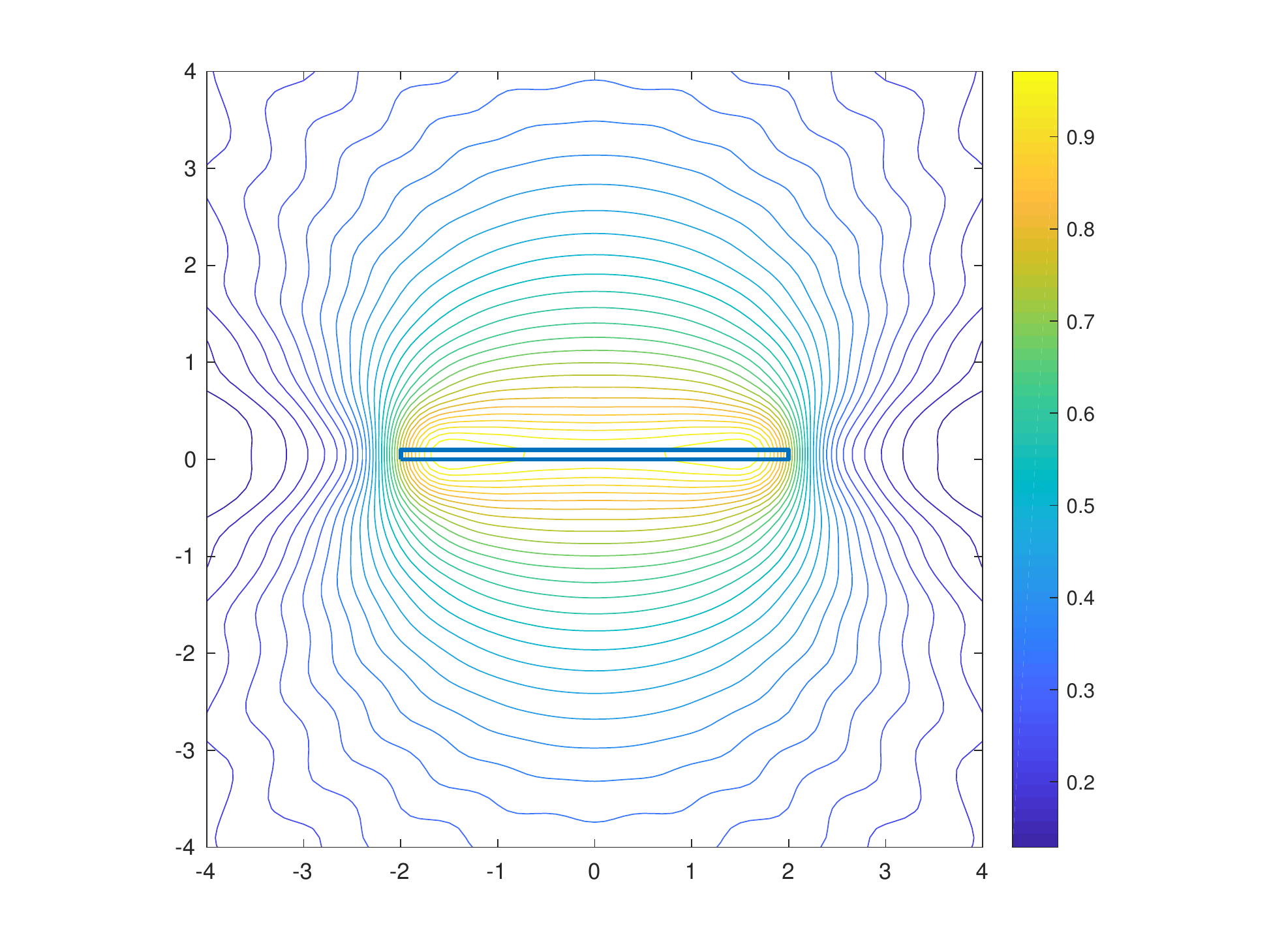}}
\end{tabular}
\end{center}
\caption{Reconstructions of larger objects when $F(x,y) = 5$. Left: triangle. Right: thin bar.}
\label{LargeOneF2}
\end{figure}

\section*{Acknowledgement}
The research of G. Hu is supported by the NSFC grant (No. 11671028) and NSAF grant (No. U1530401).
The research of X. Liu is supported by the NNSF of China under grant 11571355 and the Youth Innovation Promotion Association, CAS.
The research of J. Sun is partially supported by NSF DMS-1521555 and NSFC Grant (No. 11771068)


\begin{thebibliography}{99} \label{bbiibb}

\bibitem{BaoLinTriki}
G. Bao, J. Lin and F. Triki,
{\it  A multi-frequency inverse source problem},
 J. Differ. Equ., 249 (2010), 3443-3465.

\bibitem{BaoLuRundellXu}
G. Bao, S. Lu, W. Rundell, and B. Xu,
{\it A recursive algorithm for multi-frequency acoustic inverse source problems},
SIAM J. Numer. Anal., 53 (2015), 1608-1628.

\bibitem{Dassios}
G. Dassios,
{\it Electric and magnetic activity of the brain in spherical and ellipsoidal geometry},
Mathematical Modeling in Biomedical Imaging I: Electrical and Ultrasound Tomographies, Anomaly Detection, and Brain Imaging,
(H. Ammari eds), Springer-Verlag, Berlin, Heidelberg, (2009), 133-202.

\bibitem{DevaneyMarengoLi}
A. Devaney, E. Marengo and M. Li,
{\it The inverse source problem in nonhomogeneous background media},
 SIAM J. Appl. Math., 67 (2007), 1353-1378.

\bibitem{DevaneySherman}
A. Devaney and G. Sherman,
{\it Nonuniqueness in inverse source and scattering problems},
IEEE Trans. Antennas Propag., 30 (1982), 1034-1037.


\bibitem{El-Badia}
A. El-Badia and T. Ha-Duong,
{\it An inverse source problem in potential analysis},
 Inverse Problems, 16 (2000), 651-664.

\bibitem{El-Badia2} A. El-Badia and T. Ha-Duong,
{\it On an inverse source problem  for the heat equation. Application to a pollution detection problem.}
J. Inverse Ill-Posed Problems, 10 (2002), 585-99.

\bibitem{EllerValdivia}
M. Eller and N. P. Valdivia,
{\it Acoustic source identification using multiple frequency information},
Inverse Problems, 25 (2009), 115005.

\bibitem{JiLiuXi}
X. Ji, X. Liu and Y. Xi,
{\it Direct sampling methods for inverse elastic scattering problems},
 arXiv:1711.00626, 2017.

 \bibitem{Griesmaier2011IP} R. Griesmaier,
 	{\it Multi-frequency orthogonality sampling for inverse obstacle scattering problems},
	Inverse Problems, 27 (2011), 085005.

\bibitem{Griesmaier} R. Griesmaier and C. Schmiedecke.
 {\it A factorization method for multi-frequency inverse source problems with spars far field measurements},
 SIAM J. Imag. Sci., (2017), to appear.

 \bibitem{LiuIP17} X. Liu,
	{\it A novel sampling method for multiple multiscale targets from scattering amplitudes at a fixed frequency},
	Inverse Problems, 33 (2017), 085011.
\bibitem{LiuSun2014IP} X. Liu and J. Sun,
	{\it Reconstruction of Neumann eigenvalues and support of sound hard obstacles},
	Inverse Problems, 30 (2014), 065011.
\bibitem{Potthast2010IP} R. Potthast,
	{\it A study on orthogonality sampling},
	Inverse Problems, 26 (2010), 074015.

\bibitem{Schuhmacher} A. Schuhmacher, J. Hald, K.B. Rasmussen and P.Hansen,
	{\it Sound source reconstruction using inverse boundary element calculations},
	J. Acoust. Soc. Am., 113 (2003), 114-27.
\bibitem{Sylvester}  J. Sylvester and J. Kelly,
   {\it A scattering support for broadband sparse far field measurements },
   Inverse Problems, 21 (2005), 759-771.
\bibitem{WangEtal2017IP} X. Wang, Y. Guo, D. Zhang and H. Liu,
	{\it Fourier method for recovering acoustic sources from multi-frequency far-field data},
	Inverse Problems, 33 (2017), 035001.
\bibitem{Sun2012IP} J. Sun,
	{\it An eigenvalue method using multiple frequency data for inverse scattering problems,}
	Inverse Problems, 28 (2012), 025012.
\bibitem{ZhangSun2015JCP} R. Zhang and J. Sun,
	{\it Efficient finite element method for grating profile reconstruction},
	J. Comput. Phys., 302 (2015), 405-419.
   \end{thebibliography}
  \end{document}